\documentclass[12pt,a4paper]{article}

\usepackage{a4wide}

\usepackage[latin1]{inputenc}
\usepackage[T1]{fontenc}
\usepackage[english]{babel}

\usepackage{amsmath}
\usepackage{amssymb}
\usepackage{amsfonts,makeidx,amsthm,mathrsfs}

\newcommand{\F}{\mathscr{F}}

\newcommand{\Lr}{\mathscr{L}}
\newcommand{\dvol}{\frac{\omega^n}{n!}}
\newcommand{\dvollin}{\frac{\omegalin^n}{n!}}

\newcommand{\dsubvol}{\frac{\omega^{n-1}}{(n-1)!}}
\newcommand{\dsubsubvol}{\frac{\omega^{n-2}}{(n-2)!}}

\newcommand{\tr}{\mathrm{tr}}
\newcommand{\omegalin}{\omega_{\textrm{lin}}}

\newcommand{\Walg}{\mathbb{W}}
\newcommand{\Sp}{\mathrm{Sp}}

\newcommand{\R}{\mathbb{R}}
\newcommand{\C}{\mathbb{C}}
\newcommand{\N}{\mathbb{N}}
\newcommand{\Z}{\mathbb{Z}}

\def\cl{\mathop{\mathrm{Cl}}\nolimits}

\def\ham{\mathop{\mathrm{Ham}}\nolimits}

\newcommand{\A}{\mathcal{A}}
\newcommand{\V}{\mathcal{V}}
\newcommand{\T}{\mathcal{T}}
\newcommand{\E}{\mathcal{E}}
\newcommand{\J}{\mathcal{J}}

\newcommand{\D}{\mathcal{D}}
\newcommand{\RR}{\mathcal{R}}
\newcommand{\Rr}{\mathrm{R}}

\newcommand{\extwedge}{\stackrel{\circ}{\wedge}}
\newcommand{\WW}{\mathbb{W}}
\newcommand{\ddt}{\frac{d}{dt}}
\newcommand{\dds}{\frac{d}{ds}}
\newcommand{\ddto}{\left.\ddt\right|_{0}}
\newcommand{\ddso}{\left.\dds\right|_{0}}

\newcommand{\invnu}{\frac{1}{\nu}}

\newcommand{\W}{\mathcal{W}}

\newtheorem{theorem}{Theorem}[section]

\newtheorem{lemma}[theorem]{Lemma}
\newtheorem{cor}[theorem]{Corollary}
\newtheorem{prop}[theorem]{Proposition}
\newtheorem{defi}[theorem]{Definition}

\theoremstyle{definition} 
\newtheorem{ex}[theorem]{Exemple}

\theoremstyle{remark}

\newtheorem{rem}[theorem]{Remark}

\begin{document}

\renewcommand{\refname}{Bibliography}

\title{The formal moment map geometry of the space of symplectic connections}

\author{Laurent La Fuente-Gravy\\
	\scriptsize{laurent.lafuente@uni.lu}\\
	\footnotesize{Universit\'e du Luxembourg}\\[-7pt]
	\footnotesize{Luxembourg} \\[-7pt]}

\maketitle

\begin{abstract}
We deform the moment map picture on the space of symplectic connections on a symplectic manifold. To do that, we study a vector bundle of Fedosov star product algebra on the space of symplectic connections $\mathcal{E}$. We describe a natural formal connection on this bundle adapted to the star product algebras on the fibers. We study its curvature and show the $*$-product trace of the curvature is a formal symplectic form on $\E$. The action of Hamiltonian diffeomorphisms on $\E$ preserves the formal symplectic structure and we show the $*$-product trace can be interpreted as a formal moment map for this action. Finally, we apply this picture to study automorphisms of star products and Hamiltonian diffeomorphisms.
\end{abstract}

\noindent {\footnotesize {\bf Keywords:} Symplectic connections, Moment map, Deformation quantization, Hamiltonian diffeomorphisms, Hamiltonian automorphisms.\\
{\bf Mathematics Subject Classification (2010):}  53D55, 53D20, 32Q15, 53C21}

\tableofcontents


\section{Introduction}

We deform the moment map picture on the space of symplectic connections \cite{cagutt} on a closed symplectic manifold. We obtain a formal symplectic form on the space of symplectic connections and derive a formal moment map for the action of the group of Hamiltonian diffeormorphisms. We then relate the underlying action homomorphism from Shelukhin \cite{Shel} to an invariant from \cite{FutLLF2}.

Our formal moment map picture comes from deformation quantization \cite{BFFLS}. It is a formal counterpart of the work of Foth--Uribe \cite{FU}, who showed that the trace of Berezin-Toeplitz operators can be interpreted as a moment map on the space of almost complex structure on an integral symplectic manifold.

We use the notion of formal connections introduced by Andersen--Masulli--Sch\"atz \cite{AMS} We consider a vector bundle of formal star product algebras on the space of symplectic connections. Topologically, it is simply a product with fiber the space of formal smooth functions on $M$. Algebraically, on the fiber over the symplectic connection $\nabla$, the star product is $*_{\nabla}$, the Fedosov star product obtained with symplectic connection $\nabla$ and a trivial formal series of closed $2$-form. Working as in Masulli's Thesis \cite{MasulliThesis}, we define a canonical connection by interpreting the classical construction of isomorphisms of Fedosov star product algebra's \cite{fed2} as a parallel transport. We recover the explicit formula for the curvature and go one step further showing the curvature $2$-form acts by inner derivations, i.e the star commutator with a formal function.

In the approach of \cite{FU}, a certain determinant line bundle is considered and its curvature produces a symplectic form on the space of almost complex structures. With our infinite dimensional fibers, we bypass the determinant line bundle step and we consider directly the star product trace of the curvature element. It gives a formal symplectic form on the space of symplectic connections, i.e. a series of closed $2$-forms whose leading term is non-degenerate. We show that this $2$-form is a deformation of the symplectic structure on the space of symplectic connections and that the action of the group of Hamiltonian diffeomorphism preserves this formal symplectic form. With this picture, analogous to \cite{FU}, we prove the star product trace is a formal moment map. Hence zeroes of the formal moment map is a symplectic connection giving rise to a closed Fedosov star product.

Finally, we propose the following application to the study of Hamiltonian diffeomorphisms group. We show that the parallel lift of a Hamiltonian diffeomorphisms path gives Hamiltonian automorphisms of the star product \cite{LLF0}. Next, we apply the construction of Shelukhin \cite{Shel}, to derive a Weinstein action homomorphism which is an invariant on the $\pi_1$ of the group of Hamiltonian diffeomorphisms. On K\"ahler manifold, when restricting this action homomorphism to Hamiltonian biholomorphisms we recover the invariant from \cite{FutLLF2} obstructing the closedness of the Fedosov star product $*_{\nabla}$.

The motivation of this paper is twofold. First, it justifies the study of closed Fedosov star product, see \cite{LLF,LLF2}, Futaki--Ono \cite{FO} and Futaki--L. \cite{FutLLF1,FutLLF2}, as a zero moment map problem, a picture that was up to now only valid at first orders in the formal parameter. Second, there is an increasing interest in the study of the quantization of Hamiltonian diffeomorphisms from the point of view of symplectic topology, see for example Ioos \cite{Ioos, Ioos2}, Charles \cite{Ch}, Charles--Polterovich \cite{ChPolte} and Charles--Le Floch \cite{ChLeFloch} (in the past 3 years!). This paper proposes a formal approach to study the Hamiltonian diffeomorphisms group via quantization.


\section{Formal connections}

We explain, in our context, the work of Andersen--Masulli--Sch\"atz \cite{AMS} and Masulli \cite{MasulliThesis} who constructed formal connections adapted to families of Fedosov star products. The only difference is that we are dealing with a bundle over the whole infinite dimensional space of symplectic connections and that we are only considering Fedosov star products build with a trivial choice of series of closed $2$-form.

\subsection{Definitions}

Throughout the paper, we consider a closed symplectic manifold $(M,\omega)$. Let us briefly recall the symplectic structure of the space of symplectic connections. A connection $\nabla$ on the tangent bundle $TM$ is called \emph{symplectic} if $\nabla \omega=0$ and $\nabla$ has no torsion. There always exists a symplectic connection $\nabla$ on a sympletic manifold and the space of symplectic connection $\E(M,\omega)$ is the affine infinite dimensional space
$$\E(M,\omega)=\{\nabla+A \,|\, A\in \Lambda^1(M)\otimes \textrm{End}(TM,\omega) \textrm{ such that } \underline{A}:=\omega(\cdot,A(\cdot)\cdot)\in \Gamma(S^3T^*M) \}$$
where $\Gamma(S^3T^*M)$ is the space of completely symmetric $3$-tensors on $M$. The symplectic form on $\E(M,\omega)$ is defined by
$$\Omega^{\E}_{\nabla}(A,B):=\int_M \Lambda^{i_1j_1}\Lambda^{i_2j_2}\Lambda^{i_3j_3}\underline{A}_{i_1i_2i_3}\underline{B}_{j_1j_2j_3} \dvol,$$
where we identify tangent vectors $A, B$ to elements $\underline{A}, \underline{B}$ in $\Gamma(S^3T^*M)$, $\Lambda^{kl}$ is the inverse of the matrix of $\omega_{kl}$ in coordinates and we use summation convention on repeated indices. The group $\ham(M,\omega)$ of Hamiltonian diffeomorphisms acts symplectically on $\E(M,\omega)$ by
\begin{equation*} \label{eq:action}
(\varphi.\nabla)_X Y := \varphi_*(\nabla_{\varphi^{-1}_* X}\varphi^{-1}_* Y),
\end{equation*}
for $\varphi\in \ham(M,\omega)$, for all $X,Y \in TM$ and $\nabla \in \E(M,\omega)$. This action admits a moment map \cite{cagutt}, which will be recovered later from our formal moment map picture.

Let us now consider the attached Poisson algebra of smooth functions $(C^{\infty}(M),\{\cdot,\cdot\},\cdot)$, where $\{F,G\}:=-\omega(X_F,X_G)$ for $X_F$ defined by $\imath(X_F)\omega=dF$ with $F,G \in C^{\infty}(M)$. A \emph{star product} is an associative product on the space $C^{\infty}(M)[[\nu]]$ of formal power series of smooth functions:
$$F*G:=\sum_{r=0}^{+\infty} \nu^r C_r(F,G),$$
given by bidifferential $\nu$-linear operators $C_r$, such that $C_0(F,G)=FG$, $C_1^-(F,G):=C_1(F,G)-C_1(G,F)=\{F,G\}$ and $F*1=F=1*F$.

In this paper, we focus on Fedosov star products, which exist on any symplectic ma\-nifold. Namely, given a symplectic connection $\nabla$ and a series $\Omega$ of closed $2$-forms on $M$, Fedosov \cite{fed2} gives a geometric construction of a star product $*_{\nabla,\Omega}$. Fedosov's procedure will be recalled in the next Subsection. Except for some remarks, we will be concerned with the Fedosov star product $*_{\nabla}:=*_{\nabla,0}$ using a trivial series $\Omega=0$.

As in \cite{AMS}, to a family of star product $\{*_{\sigma}\}_{\sigma\in \mathcal{T}}$ on $(M,\omega)$ parametrized by a smooth manifold $\mathcal{T}$, one associate the vector bundle 
\begin{equation*}
\V(\T,\{*_{\cdot}\}):=\T\times C^{\infty}(M)[[\nu]] \stackrel{p}{\longrightarrow} \T,
\end{equation*}
with $p$ the projection on the first factor. While topologically, the bundle is trivial, the algebraic structure on the fiber depends on the basepoint as we endow $p^{-1}(\sigma)$ with the star product $*_{\sigma}$.

A {\it formal connection} $\D$ is a connection in the bundle $\V(\T,\{*_{\cdot}\})$ that can be written as 
\begin{equation}\label{eq:Ddef}
\D_X F:= X(F) + \beta(X)\, F,
\end{equation}
where $F$ is a section of $\V(\T,\{*_{\cdot}\})$, $X$ is a tangent vector to $\T$ and $\beta(X)$ is a series
$$\beta(X):=\sum_{k=1}^{\infty} \nu^k \beta_k(X),$$
for $\beta_k(\cdot)$ being smooth $1$-forms on $\T$ with values in the space of differential operators over $M$ (note that $\beta(X)=0\, (\textrm{mod }\nu)$). The formal connection is called {\it compatible} if for any smooth sections $F,G$ of $\V(\T,\{*_{\cdot}\})$, we have
$$\D(F*_{\sigma}G)=\D(F)*_{\sigma} G + F*_{\sigma} \D(G).$$

\begin{defi}
We call the \emph{tautological family of star products}, the family of Fedosov star products $\{*_{\nabla}\}_{\nabla\in \E(M,\omega)}$ parametrized by the symplectic manifold $\E(M,\omega)$ of all symplectic connections. To this family we associate the \emph{tautological vector bundle} $\mathcal{V}:= \mathcal{V}(\E(M,\omega),\{*_{\cdot}\})$ with fiber $p^{-1}(\nabla)$ endowed with the Fedosov star product $*_{\nabla}$. 
\end{defi}

\noindent In the Subsection \ref{subsect:canonical}, we will recall the construction from \cite{MasulliThesis} of a canonical formal connection on $\V$.

Remark that the base manifold and the fibers of $\V$ are both of infinite dimension. The space of symplectic connection which is modeled on $\Gamma(S^3T^{*}M)$ is a Fr\'echet manifold as well as $C^{\infty}(M)$ as we are working on a closed manifold. Moreover, if one consider the space $\R^{\N}$ of sequences of real numbers, one defines semi-norms $||a_{\cdot}||_j:=|a_j|$ for $j\in\N$ and $a_{\cdot}\in \R^{\N}$, the topology induced by the semi-norms makes $\R^{\N}$ Fr\'echet and smooth functions from $M$ to $\R^{\N}$ are identified with formal power series $C^{\infty}(M)[[\nu]]$. In the Fedosov settings, we will also meet formal power series of symmetric tensors on $M$ which are products of Fr\'echet spaces. So that we can perform differential geometry in our context, for details about differential geometry in infinite dimensional settings we refer to the survey \cite{KHN}.

\subsection{Fedosov star products}

We briefly recall Fedosov's construction of star product and isomorphisms of Fedosov star products \cite{fed2}.

On the symplectic manifold $(M,\omega)$, let $x\in M$ and $\{e_1,\ldots,e_{2n}\}$ be a basis of $T_xM$ and $\omega_{ij}:=\omega(e_i,e_j)$. We consider the dual basis $\{y^1,\ldots, y^{2n}\}$ of $T^*_xM$. The formal Weyl algebra $\WW_x$ at the point $x$ is the set of formal power series of symmetric forms on $T_xM$, that is elements
$$a(y,\nu):=\sum_{2k+r=0}^{\infty} \nu^r a_{r,i_1\ldots i_k}y^{i_1}\ldots y^{i_k}$$
for $a_{r,i_1\ldots i_k}$ symmetric in $i_1\ldots i_k$ and $2k+r$ is called the total degree, endowed with the $\circ$-product defined by
\begin{eqnarray*}
a(y,\nu)\circ b(y,\nu) & := &\left.\left( \exp\left(\frac{\nu}{2}\Lambda^{ij} \partial_{y^i} \partial_{z^j}\right)a(y,\nu)b(z,\nu)\right)\right|_{y=z} \nonumber \\
\end{eqnarray*}

We form the formal Weyl algebra bundle $\W:=\bigsqcup_{x\in M}\WW_x$, as well as the bundle of differential forms with values in the Weyl algebra bundle by $\W\otimes \Lambda M$ whose sections are locally of the form:
\begin{equation} \label{eq:sectionofW}
\sum_{2k+l\geq 0,\, k,l\geq 0, p\geq 0} \nu^k a_{k,i_1\ldots i_l,j_1\ldots j_p}(x)y^{i_1}\ldots y^{i_l}dx^{j_1}\wedge \ldots \wedge dx^{j_p}.
\end{equation}
The $a_{k,i_1\ldots i_l,j_1\ldots j_p}(x)$ are, in the indices $i_1,\ldots,i_l,j_1,\ldots,j_p$, the components of a tensor on $M$, symmetric in the 
$i$'s and antisymmetric in the $j$'s. We have a filtration of the space of sections $\Gamma \W\otimes \Lambda^*(M)$ with respect to the total degree
$$\Gamma \W\otimes \Lambda^*(M)\supset \Gamma \W^1\otimes \Lambda^*(M)\supset \Gamma \W^2\otimes \Lambda^*(M)\supset \ldots.$$
The space $\Gamma \W\otimes \Lambda^*(M)$ is made into an algebra by extending the $\circ$-product, for $a, b\in \Gamma \W$ and $\alpha, \beta\in \Omega^*(M)$, we define
$(a\otimes \alpha) \circ (b\otimes \beta) := a\circ b \otimes \alpha\wedge \beta$, where $\circ$ is performed fiberwise. The graded commutator $[s,s']:=s\circ s'- (-1)^{q_1q_2}s'\circ s$ where $s$ is a $q_1$-form and $s'$ a $q_2$-form (anti-symmetric degree), gives to $\W$-valued forms  the structure of a graded Lie algebra.

A symplectic connection $\nabla$ on $(M,\omega)$ induces a derivation $\partial$ of degree $+1$ on $\W$-valued forms by :
\begin{equation*}
\partial a := da + \frac{1}{\nu}[\overline{\Gamma},a] 
\end{equation*}
for $a\in \Gamma \W\otimes \Lambda(M)$, where $\overline{\Gamma}:=\frac{1}{2}\omega_{lk}\Gamma^k_{ij}y^ly^jdx^i$, for $\Gamma^k_{ij}$ the Christoffel symbols of $\nabla$ (note that $\omega_{lk}\Gamma^k_{ij}$ is symmetric $l,j$ because $\nabla$ preseves the symplectic form). 

The curvature of $\partial$ is
\begin{equation*}
\partial\circ \partial\, a := \frac{1}{\nu}[\overline{R},a],
\end{equation*}
where $\overline{R}:= \frac{1}{4} \omega_{ir}R^r_{jkl}y^iy^jdx^k\wedge dx^l$, for $R^r_{jkl}:=\left(R(\partial_k,\partial_l)\partial_j\right)^r$ the components of the curvature tensor of $\nabla$.

To make this connection flat, we consider connections on $\Gamma \W$ of the form
\begin{equation} \label{eq:defD}
D a:=\partial a - \delta a + \frac{1}{\nu}[r,a],
\end{equation}
where $r$ is a $\W$-valued $1$-form and $\delta$ is defined by
\begin{equation*} \label{eq:deltadef}
\delta(a) := dx_k\wedge \partial_{y_k} a=-\frac{1}{\nu}[\omega_{ij}y^i dx^j,a].
\end{equation*}
The curvature of $D$ is
\begin{equation*}
D^2 a = \frac{1}{\nu}\left[\overline{R} + \partial r - \delta r + \frac{1}{2\nu}[r,r]-\omega,a\right].
\end{equation*}
So, $D$ is flat if one can find a $\W$-valued $1$-form $r$ so that
\begin{equation}\label{eq:req}
\overline{R} + \partial r - \delta r + \frac{1}{\nu}r\circ r  = \Omega,
\end{equation}
for $\Omega \in \nu\Omega^2(M)[[\nu]]$ being any closed formal $2$-form.

Define 
$$\delta^{-1} a_{pq}:= \frac{1}{p+q}y^ki(\partial_{x^k})a_{pq} \textrm{ if } p+q>0 \textrm{ and } \delta^{-1}a_{00}=0,$$
where $a_{pq}$ is a $q$-forms with $p$ $y$'s and $p+q>0$. Fedosov showed \cite{fed2}, for any given closed central $2$-form $\Omega$,
there exists a unique solution $r \in \Gamma \W \otimes \Omega^1 M$ with $\W$-degree at least $3$ of:
\begin{equation*}
\overline{R} + \partial r - \delta r + \frac{1}{\nu}r\circ r = \Omega,
\end{equation*}
and satisfying $\delta^{-1}r=0$. 

We denote by $D$ the flat connection defined by \eqref{eq:defD} with $r$ the unique solution of \eqref{eq:req}. Define $\Gamma \W_{D} := \{a\in \Gamma \W | D a=0\}$ the space of flat sections, it is an algebra for the $\circ$-product since $D$ is a derivation. Define also the symbol map $\sigma :a\in \Gamma \W_{D} \mapsto \left.a\right|_{y=0}\in C^{\infty}(M)[[\nu]]$. Fedosov showed \cite{fed2} that $\sigma$ is a bijection with inverse $Q$ defined by 
\begin{equation*} \label{eq:defQ}
Q:=\sum_{k\geq 0} \left(\delta^{-1}(\partial + \frac{1}{\nu}[r,\cdot])\right)^k.
\end{equation*}
Via $\sigma$ and $Q$, the $\circ$-product induces a star product $*$ on $C^{\infty}(M)[[\nu]]$ that is called a Fedosov star product.

Later, we will need the following technical lemma stating that $D$-flat $1$-forms are also $D$-exact with canonical representative.
\begin{lemma}\label{lemme:Dinverse}
Suppose $b \in \Gamma(W)\otimes \wedge^1(M)$ satisfy $Db = 0$. Then the equation $Da = b$ admits a
unique solution $a \in \Gamma(W)$, such that $a|_{y=0} = 0$, it is given by
$$b=D^{-1}a:=-Q(\delta^{-1}a).$$
\end{lemma}

In the sequel, to emphasize the dependence of $*$ (resp. $r$, $D$ and $Q$) in the choices $\nabla$ and $\Omega$, we will write $*_{\nabla,\Omega}$ (resp. $r^{\nabla,\Omega}$, $D^{\nabla,\Omega}$ and $Q^{\nabla,\Omega}$) and simply $*_{\nabla}$ (resp. $r^{\nabla}$,$D^{\nabla}$ and $Q^{\nabla}$) when $\Omega=0$.

We now describe a canonical way to lift smooth path of symplectic connections to isomorphisms of Fedosov star product algebra. We restrict our attention to star products of the form $*_{\nabla}$.

To do that we need to consider sections of the extended bundle $\W^+\supset \W$ which are locally of the form 
\begin{equation*} \label{eq:sectionofW+}
\sum_{2k+l\geq 0, l\geq 0} \nu^k a_{k,i_1\ldots i_l}(x)y^{i_1}\ldots y^{i_l}.
\end{equation*}
similar to \eqref{eq:sectionofW}, with $p=0$, but we allow $k$ to take negative values, the total degree $2k+l$ of any term must remain nonnegative and in each given nonnegative total degree there is a finite number of terms.

\begin{theorem}[Fedosov \cite{fed2}] \label{theor:smoothisom}
Consider the smooth path $t\in [0,1] \mapsto \nabla^t\in \E(M,\omega)$, then there exists isomorphisms $B_t:\Gamma\W_{D^{\nabla^0}} \rightarrow \Gamma \W_{D^{\nabla^t}}$. Moreover, $B_t$ can be chosen in the canonical form
$$B_t a:=v_t\circ a \circ v_t^{-1}$$
for $v_t \in \Gamma \W^+$ is the unique solution of the initial value problem:
\begin{equation} \label{eq:vt}
\left\{\begin{array}{rcl}
\frac{d}{dt}v_t & = & \frac{1}{\nu}h_t\circ v_t \\
v_0 &= & 1
\end{array}\right.
\end{equation}
with 
\begin{equation} \label{eq:ht}
h_t:=-(D^{\nabla^t})^{-1}\left(\ddt \overline{\Gamma}^{\nabla^t}+\ddt r^{\nabla^t}\right).
\end{equation}
\end{theorem}

\noindent We reproduce the proof as our results rely on the above Theorem where $B_t$ will be interpreted as a parallel transport above the curve $t \mapsto \nabla^t$.

\begin{proof}
We first show that $h_t$ is well-defined, that is we compute $D^{\nabla^t}\left(\ddt \overline{\Gamma}^{\nabla^t}+\ddt r^{\nabla^t}\right)=0$. Write $D^{\nabla^0}=D^{\nabla^t}-\frac{1}{\nu} [\Delta\Gamma_t+\Delta r_t,\cdot]$ for $\Delta\Gamma_t=\overline{\Gamma}^{\nabla^t}-\overline{\Gamma}_0$ and $\Delta r_t=r^{\nabla^t}-r^{\nabla^0}$. Computing the curvatures, we obtain
$$-D^{\nabla^t}\left(\Delta\Gamma_t+\Delta r_t\right)+\frac{1}{\nu} \left(\Delta\Gamma_t+\Delta r_t\right)^2=0.$$
Differentiating with respect to $t$, we get
$$-(\ddt D^{\nabla^t})\left(\Delta\Gamma_t+\Delta r_t\right)-D^{\nabla^t} \left(\ddt \overline{\Gamma}^{\nabla^t}+\ddt r^{\nabla^t}\right) + \frac{1}{\nu}[\ddt \overline{\Gamma}^{\nabla^t}+\ddt r^{\nabla^t}, \Delta\Gamma_t+\Delta r_t]=0.$$
Since $\ddt D^{\nabla^t}=\frac{1}{\nu}[\ddt \overline{\Gamma}^{\nabla^t}+\ddt r^{\nabla^t},\cdot]$, we have
$$D^{\nabla^t} \left(\ddt \overline{\Gamma}^{\nabla^t}+\ddt r^{\nabla^t}\right)=0.$$
So that $h_t$ is well defined by Lemma \ref{lemme:Dinverse}

After that, let us recall that the solution $v_t$ to Equation \eqref{eq:vt} is given recursively by 
\begin{equation} \label{eq:defvt}
v_t=1+\int_0^t \frac{1}{\nu} h_t\circ v_t\, dt,
\end{equation}
which is well defined in $\Gamma \W^+$ as $h_t\in \Gamma\W$ is of degree at least $3$ so that $\frac{1}{\nu} h_t\circ \cdot$ rises the $\W$-degree at least by $1$. Moreover, $v_t$ is invertible as it starts with $1$ modulo terms of higher $\W$-degree.

We conclude the proof by computing $D^{\nabla^t}(v_t\circ a \circ v_t^{-1})$ for some $a\in \Gamma \W$.
\begin{eqnarray*}
D^{\nabla^t}(v_t\circ a \circ v_t^{-1}) & = & v_t \circ[v_t^{-1}\circ D^{\nabla^t}v_t,a]\circ v_t^{-1}+ v_t\circ D^{\nabla^t}a\circ v_t^{-1}, \\
 & = & v_t \circ[v_t^{-1}\circ D^{\nabla^t}v_t+\frac{1}{\nu} \left(\Delta\Gamma_t+\Delta r_t\right) ,a]\circ v_t^{-1}+ v_t\circ D^{\nabla^0}a\circ v_t^{-1}.
\end{eqnarray*}
We now analyse the expression $v_t^{-1}\circ D^{\nabla^t}v_t+\frac{1}{\nu} \left(\Delta\Gamma_t+\Delta r_t\right)$. At $t=0$, it vanishes. Its derivative with respect to $t$ gives
\begin{eqnarray*}
\ddt \left(v_t^{-1}\circ D^{\nabla^t}v_t+\frac{1}{\nu} \left(\Delta\Gamma_t+\Delta r_t\right)\right)
& = & \invnu v_t^{-1}\circ \left( D^{\nabla^t} h_t + \invnu \left(\ddt \overline{\Gamma}^{\nabla^t}+\ddt r^{\nabla^t}\right)\right) \circ v_t\\
 & = & 0\ (\textrm{by definition of } h_t).
\end{eqnarray*}
It shows that  $v_t^{-1}\circ D^{\nabla^t}v_t+\frac{1}{\nu} \left(\Delta\Gamma_t+\Delta r_t\right)=0$. So that, $D^{\nabla^t}(v_t\circ a \circ v_t^{-1}) =v_t\circ D^{\nabla^0}a\circ v_t^{-1}$, in particular $B_t$ maps $D^{\nabla^0}$-flat sections to $D^{\nabla^t}$-flat sections, which finishes the proof.
\end{proof}

\begin{rem}
We wrote in the statement that $B_t$ is canonical, by this we mean that it is completely determined by the data of the smooth path of symplectic connections. Other automorphisms can be obtained by modifying for example $h_t$ by a flat section or by considering $\tilde{h}_t:=-(D^{\nabla^t})^{-1}\left(\ddt \overline{\Gamma}^{\nabla^t}+\ddt r^{\nabla^t}+ \theta \right)$ for any series of closed $1$-form $\theta$.
\end{rem}

\begin{rem}
One remark should be done with respect to the regularity of $t\mapsto v_t$. First, $h_t$ given by Equation \eqref{eq:ht} depends polynomially on $\ddt \overline{\Gamma}^{\nabla^t}$, its covariant derivatives and the curvature of $\nabla^t$, so that $t \mapsto h_t$ is a smooth path in $\Gamma \W^3$, that is with total degree greater or equal to $3$. In Equation \eqref{eq:defvt} defining $v_t$, we integrate, in each degree, a smooth path that is a finite sum of symmetric tensors. Moreover, it is a classical result \cite{fed2}, that $v_t$ can be realised as $\exp(\invnu s_t)$ for $s_t \in \Gamma \W^3$ and for which the term of $\W$-degree $k$ in $s_t$ depends polynomially in the lower degree terms of $h_t$. It means the path $t\mapsto v_t$ is smooth.
\end{rem}

\begin{rem}
For $v_t\in \Gamma \W^+$ solution of Equation \eqref{eq:vt}, we say that $v_t$, as well as $v_1$ are \emph{generated} by $h_t$.
\end{rem}

\subsection{A canonical compatible formal connection on $\V$} \label{subsect:canonical}

We give a compatible formal connection on $\V$ following the construction in \cite{MasulliThesis}. The idea is to consider the automorphisms $B_t$ from Theorem \ref{theor:smoothisom} as a parallel transport along a path of symplectic connections.

That is, given a smooth path $t\in [0,1]\mapsto \nabla^t$ and $F\in C^{\infty}(M)[[\nu]]$, we want a formal connection $\D$ so that 
$$\D_{\ddt \nabla^t} \left(\left.\left(v_t\circ Q^{\nabla^0}(F)\circ v_t^{-1}\right)\right|_{y=0}\right)=0,$$
for $v_t$ generated by $h_t$ as in Theorem \ref{theor:smoothisom}, where we see $\left.\left(v_t\circ Q^{\nabla^0}(F)\circ v_t^{-1}\right)\right|_{y=0}$ as a section of the restriction of $\V$ to $[0,1]$. If we write $\D_{\ddt \nabla^t}=\ddt +\beta(\ddt \nabla_t)$ as in Equation \eqref{eq:Ddef}, we obtain
\begin{equation} \label{eq:beta}
\left.\invnu[h_t,v_t\circ Q^{\nabla^0}(F)\circ v_t^{-1}]\right|_{y=0} + \beta(\ddt \nabla^t)\left.\left(v_t\circ Q^{\nabla^0}(F)\circ v_t^{-1}\right)\right|_{y=0}=0.
\end{equation}

\begin{defi} \label{def:formalconn}
Define, for $A\in T_{\nabla}\E(M,\omega)$,
\begin{itemize}
\item the \emph{connection $1$-form} $\alpha\in \Omega^1(\E(M,\omega),\Gamma\W^3)$ by
\begin{equation} \label{eq:defalpha}
\alpha_{\nabla}(A):=(D^{\nabla})^{-1}\left(\overline{A}+ \ddt r^{\nabla+tA}\right),
\end{equation}
for $\overline{A}:=\frac{1}{2}\omega_{lk}A_{ij}^ky^ly^jdx^i=\ddt \overline{\Gamma}^{\nabla+tA}$,
\item the $1$-form $\beta$ with values in formal differential operators:
\begin{equation*}
\beta_{\nabla}(A)(F):= \left.\invnu[\alpha_{\nabla}(A),Q^{\nabla}(F)]\right|_{y=0},\,\textrm{ for }F\in C^{\infty}(M)[[\nu]],
\end{equation*}
\item the \emph{formal connection} $\D:= d+\beta$
\end{itemize}
\end{defi}  

\noindent As by Theorem \ref{theor:smoothisom}: $v_t\circ Q^{\nabla^0}(F)\circ v_t^{-1}=Q^{\nabla^t}\left(\left.\left(v_t\circ Q^{\nabla^0}(F)\circ v_t^{-1}\right)\right|_{y=0}\right)$, then Equation \eqref{eq:beta} is satisfied because
\begin{equation*}\label{eq:defA}
\alpha_{\nabla^t}(\ddt \nabla^t)=-h_t
\end{equation*}

As we will mainly work inside the Fedosov framework, we propose a formula for $Q^{\nabla}(\D_AF(\nabla))$ for a section $F$ of $\V$, and $A$ a tangent vector at $\nabla$. It can be interpreted as a connection on a bundle of flat sections of Weyl algebra, that is the subbundle
$$\V^{\W}\subset \E(M,\omega)\times \Gamma\W \stackrel{q}{\longrightarrow} \E(M,\omega).$$
for $q$ the projection on the first factor and $\V^{\W}$ consists of elements $(\nabla,a)\in \E(M,\omega)\times \Gamma\W$ such that $D^{\nabla}a=0$.

\begin{lemma} \label{lemme:QDF}
For $F(\cdot)$ a section of $\V$ and $A\in T_{\nabla}\E(M,\omega)$ with $\nabla\in \E(M,\omega)$, we have
\begin{equation} \label{eq:QDF}
Q^{\nabla}((\D_AF)(\nabla))= A \left(Q^{\cdot}(F(\cdot)\right)+ \invnu[\alpha_{\nabla}(A),Q^{\nabla}(F)],
\end{equation}
where $A \left(Q^{\cdot}(F(\cdot)\right):= \left.\ddt\right|_{0} Q^{\nabla+tA}(F(\nabla+tA))$ is the standard notation.
\end{lemma}

\begin{proof}
Recall the formula $\left.\ddt\right|_{0} Q^{\nabla+tA}(G)=-(D^{\nabla})^{-1}(\invnu [\overline{A}+\left.\ddt\right|_{0}r^{\nabla+tA},Q^{\nabla}(G)])$ with $G\in C^{\infty}(M)[[\nu]]$ which may be found in \cite{fed4}. Hence, 
$$ \left.\ddt\right|_{0} Q^{\nabla+tA}(F(\nabla+tA))= Q^{\nabla}\left(\left.\ddt\right|_{0} F(\nabla+tA)\right)-(D^{\nabla})^{-1}\left(\invnu [\overline{A}+\left.\ddt\right|_{0}r^{\nabla+tA},Q^{\nabla}(F(\nabla))]\right).$$

Evaluating both sides of \eqref{eq:QDF} at $y=0$,  since $-(D^{\nabla})^{-1}(\invnu [\overline{A}+\left.\ddt\right|_{0}r^{\nabla+tA},Q^{\nabla}(F(\nabla))])$ do not contribute, one recovers the definition of $\D$. It then remains to show the RHS is a flat section, which amounts to the computation of
\begin{eqnarray*}
D^{\nabla}\left( \invnu[\alpha_{\nabla}(A),Q^{\nabla}(F(\nabla))]-(D^{\nabla})^{-1}(\invnu [\overline{A}+\left.\ddt\right|_{0}r^{\nabla+tA},Q^{\nabla}(F(\nabla))])\right).
\end{eqnarray*}
Using that $D^{\nabla}$ is a derivation, the above expression becomes
$$ \invnu[D^{\nabla}\alpha_{\nabla}(A),Q^{\nabla}(F)]  - \invnu [\overline{A}+\left.\ddt\right|_{0}r^{\nabla+tA},F(\nabla)]=0,$$
by definition of $\alpha$. It finishes the proof.
\end{proof}

\begin{theorem}
The formal connection $\D$ on $\V$ is compatible with the tautological family of star products.
\end{theorem}

\begin{proof}
We consider $F,G \in C^{\infty}(M)[[\nu]]$ seen as constant sections of $\V$. Then, we have a section $\nabla \mapsto F *_{\nabla}G$ which is not constant in general. Let $A\in T_{\nabla}\E(M,\omega)$ a tangent vector. By Lemma \ref{lemme:QDF}, we have
\begin{eqnarray} \label{eq:DAF*G}
Q^{\nabla}(\D_A(F*_{\cdot}G)(\nabla)) & = & \ddto Q^{\nabla+tA}(F*_{\nabla+tA} G) + \invnu[\alpha_{\nabla}(A),Q^{\nabla}(F*_{\nabla}G)]
\end{eqnarray}
By definition of the Fedosov star product $*_{\nabla+tA}$, we have:
$$\ddto Q^{\nabla+tA}(F*_{\nabla+tA} G) =\left(\ddto Q^{\nabla+tA}(F)\right)\circ Q^{\nabla}(G)+ Q^{\nabla}(F)\circ \ddto Q^{\nabla+tA}(G).$$
On the other hand, using Lemma \ref{lemme:QDF} with constant section $F$, we get
\begin{eqnarray}
Q^{\nabla}((\D_A F)*_{\nabla} G) & = &  Q^{\nabla}(\D_AF)\circ Q^{\nabla}(G) \nonumber \\
 & = &  \left( \left.\ddt\right|_{0} Q^{\nabla+tA}(F) + \invnu[\alpha_{\nabla}(A),Q^{\nabla}(F)] \right)\circ Q^{\nabla}(G) \label{eq:DAFoG}
\end{eqnarray}
Similarly,
\begin{equation} \label{eq:FoDAG}
Q^{\nabla}(F*_{\nabla} (\D_A G)) =  Q^{\nabla}(F) \circ\left( \left.\ddt\right|_{0} Q^{\nabla+tA}(G) + \invnu[\alpha_{\nabla}(A),Q^{\nabla}(G)]\right) 
\end{equation}
Equations \eqref{eq:DAF*G}, \eqref{eq:DAFoG}, \eqref{eq:FoDAG} and Leibniz for the $\circ$-commutator, we get
$$\D_A(F*_{\cdot}G) = (\D_A F)*_{\cdot} G + F*_{\cdot} (\D_A G),$$
which means $\D$ is compatible with the tautological family of star products.
\end{proof}

\section{Curvature and a formal symplectic form on $\E(M,\omega)$}

\subsection{The curvature of $\D$}

Consider vector fields $A, B$ on $\E(M,\omega)$ and $F(\cdot)$ a section of $\V$, the curvature $\mathcal{R}$ of $\D$ is defined as
$$\mathcal{R}(A,B)F=\D_A(\D_BF)-\D_B(\D_AF)-\D_{[A,B]}(F).$$

\begin{theorem} \label{theor:Ralpha}
The curvature of $\D$ is given by
\begin{equation}\label{eq:Rdef}
\RR(A,B)F=\left.\invnu[\Rr(A,B),Q(F)]\right|_{y=0}
\end{equation}
for $\Rr(A,B)$ being the $2$-form with values in $\Gamma\W^3$ defined by
\begin{equation}\label{eq:Ralpha}
\Rr_{\nabla}(A,B):=d_{\nabla}\alpha(A,B)+\invnu[\alpha_{\nabla}(A),\alpha_{\nabla}(B)],\; \textrm{ for } \nabla\in \E(M,\omega)
\end{equation}
Moreover, 
\begin{itemize}
\item $\Rr_{\nabla}(A,B)$ is a $D^{\nabla}$-flat section,
\item $\left.\Rr_{\nabla}(A,B)\right|_{y=0}=\frac{\nu^2}{24}\Lambda^{i_1j_1}\Lambda^{i_2j_2}\Lambda^{i_3j_3}\underline{A}_{i_1i_2i_3}\underline{B}_{j_1j_2j_3} + O(\nu^3)$.
\end{itemize}
\end{theorem}

\begin{rem}
Note that the leading term of $\Rr(A,B)$ is, up to a constant, the integrand of the symplectic form $\Omega^{\E}$. It is non degenerate at every point.
\end{rem}

\begin{rem}
Most of the statement already appeared in \cite{AMS}. Our contribution is the flatness of $\Rr_{\nabla}(A,B)$ and the identification of its first term.
\end{rem}

\begin{proof}
We compute $Q^{\nabla}(\mathcal{R}(A,B)F)$ using Lemma \ref{lemme:QDF}:
\begin{eqnarray*}
Q^{\nabla}(\D_A(\D_BF)(\nabla)) & = & A_{\nabla}(Q^{\cdot}(\D_BF(\cdot)))+\invnu [\alpha_{\nabla}(A),Q^{\nabla}(\D_BF)(\nabla)] \\
 & = & A_{\nabla}(dQ^{\cdot}(F(\cdot))(B_{\cdot}))+A_{\nabla}(\invnu[\alpha_{\cdot}(B),Q^{\cdot}(F(\cdot))])\\
  & & +\invnu\left[\alpha_{\nabla}(A),B_{\nabla}(Q^{\cdot}(F(\cdot)))+\invnu[\alpha_{\nabla}(B),Q^{\nabla}(F(\nabla))]\right].\\
\end{eqnarray*}
Exchanging the roles of $A$ and $B$ we get,
\begin{eqnarray*}
Q^{\nabla}(\D_B(\D_AF)(\nabla)) 
 & = & B_{\nabla}(dQ^{\cdot}(F(\cdot))(A_{\cdot}))+B_{\nabla}(\invnu[\alpha_{\cdot}(A),Q^{\cdot}(F(\cdot))])\\
  & & +\invnu\left[\alpha_{\nabla}(B),A_{\nabla}(Q^{\cdot}(F(\cdot)))+\invnu[\alpha_{\nabla}(A),Q^{\nabla}(F(\nabla))]\right].\\
\end{eqnarray*}
The last term is 
\begin{eqnarray*}
Q^{\nabla}(\D_{[A,B]}F(\nabla)) 
 & = & dQ^{\cdot}(F(\cdot))([A,B]_{\nabla})+\invnu\left[\alpha_{\nabla}([A,B]),Q^{\nabla}(F(\nabla))\right]
\end{eqnarray*}
Using $d^2=0$, and Jacobi equation, we have
\begin{equation*}
Q^{\nabla}(\mathcal{R}(A,B)F(\nabla))=\invnu\left[d\alpha(A_{\nabla},B_{\nabla}) + \invnu[\alpha_{\nabla}(A),\alpha_{\nabla}(B)],Q^{\nabla}(F(\nabla))\right],
\end{equation*}
proving the Equations \eqref{eq:Rdef} and \eqref{eq:Ralpha}.

To compute $\Rr_{\nabla}(A,B)$ is a flat section, let us assume $A$ and $B$ are constant vector fields and denote by $D_{s,t}$ the Fedosov connection, and the element $r_{s,t}$, induced by the symplectic connection $\nabla_{s,t}:=\nabla+sA+tB$. Because $A$ and $B$ are constant, $d\alpha(A_{\nabla},B_{\nabla}) =A_{\nabla}(\alpha(B))-B_{\nabla}(\alpha(A))$. Now,
\begin{eqnarray*}
D(A_{\nabla}(\alpha(B)))& = & D(\left.\frac{d}{ds}\right|_{0}\alpha_{\nabla+sA}(B)) \\
& = & \left.\frac{d}{ds}\right|_{0}D_{s,0}\alpha_{\nabla+sA}(B)- (\ddso D_{s,0})\alpha_{\nabla}(B) \\
& = & \left.\frac{d}{dsdt}\right|_{s=t=0}r_{s,t}-\invnu[\overline{A}+\ddso r_{s,0},\alpha_{\nabla}(B)]
\end{eqnarray*}
where we used $D_{s,0}\alpha_{\nabla+sA}(B)=\overline{B}+\ddto r_{s,t}$ by definition of $\alpha$.
Similarly,
\begin{equation*}
D(B_{\nabla}(\alpha(A))) = \left.\frac{d}{dsdt}\right|_{s=t=0}r_{s,t}-\invnu[\overline{B}+\ddto r_{0,t},\alpha_{\nabla}(A)]
\end{equation*}
Finally,
\begin{equation*}
D[\alpha_{\nabla}(A),\alpha_{\nabla}(B)]=[\overline{A}+\ddso r_{s,0},\alpha_{\nabla}(B)]+\invnu[\alpha_{\nabla}(A),\overline{B}+\ddto r_{0,t}],
\end{equation*}
which shows that $D\Rr_{\nabla}(A,B)=0$ as stated.

We identify now the first order term of $\left.\Rr_{\nabla}(A,B)\right|_{y=0}$, with $A$, $B$ constant vector fields. We start with $\left.d\alpha(A,B)\right|_{y=0}$. Well,
\begin{eqnarray*}
\left.\frac{d}{ds}\right|_{0}\alpha_{\nabla+sA}(B) & = & \left.\frac{d}{ds}\right|_{0}(D_{s,0})^{-1}(\overline{B}+\ddto r_{s,t})
\end{eqnarray*}
which always contain term in $y$'s by definition of $(D_{s,0})^{-1}$. Similar computation holds for $\ddto \alpha_{\nabla+tB}(A)$, so that $\left.d\alpha(A,B)\right|_{y=0}=0$.
Finally, we compute $\left.[\alpha_{\nabla}(A),\alpha_{\nabla}(B)]\right|_{y=0}$.
Using the formula from Lemma \ref{lemme:Dinverse} for $D^{-1}$,
\begin{eqnarray*}
\alpha_{\nabla}(A) & = & \delta^{-1}(\overline{A})\, (\textrm{mod } \Gamma\W^4) \\
 & = & -\frac{1}{6} \omega_{lk}A^k_{ij} y^iy^jy^l \,(\textrm{mod } \Gamma\W^4) 
\end{eqnarray*}
Also, 
\begin{eqnarray*}
\alpha_{\nabla}(B) & = & -\frac{1}{6} \omega_{lk}B^k_{ij} y^iy^jy^l \,(\textrm{mod } \Gamma\W^4),
\end{eqnarray*}
and using that $[\cdot,\cdot]$ preserves the $\W$-degree and that the terms omitted from $\Gamma\W^4$ must contain at least one $y$, we get:
\begin{eqnarray*}
\frac{1}{\nu}\left.[\alpha_{\nabla}(A),\alpha_{\nabla}(B)]\right|_{y=0}=\frac{\nu^2}{24}\Lambda^{i_1j_1}\Lambda^{i_2j_2}\Lambda^{i_3j_3}\underline{A}_{i_1i_2i_3}\underline{B}_{j_1j_2j_3} + O(\nu^3).
\end{eqnarray*}
It finishes the proof.
\end{proof}

\subsection{A formal symplectic form on $\E(M,\omega)$}

A \emph{formal symplectic form} on a manifold $F$ is a formal power series of $2$-forms
$$\sigma:=\sigma_0+\nu\sigma_1+\ldots \in \Omega^2(F)[[\nu]],$$
which is closed and starts with a symplectic form $\sigma_0$.\\
We say $\sigma$ is a \emph{formal deformation} of $\sigma_0$.

In the paper \cite{FU}, Foth and Uribe studied a connection on a bundle of quantum spaces over the space $\J$ of almost-complex structures on an integral symplectic manifold. A similar picture is also the motivation of the introduction of formal connections in \cite{AMS}. As the quantum spaces are finite dimensional, Foth--Uribe consider the determinant line bundle with induced connection whose curvature is a $2$-form on the base manifold to study the prequantization of $\J$.

In our formal picture, we do not have quantum spaces. The fibers of our bundle $\V$ is the space $C^{\infty}(M)[[\nu]]$ which should be thought of the endomorphism (operator) space of the quantum space. A connection on the quantum space bundle induces a connection on the endomorphism bundle, from which \cite{AMS} recover a formal connection. In that settings, the curvature on the endomorphism bundle is a commutator with the curvature of the quantum space bundle.

In our formal situation, the curvature $\RR$ is the commutator with a $2$-form $\Rr$ with values in flat sections. We advocate this $\Rr$ plays the role of the curvature of the quantum space bundle. While we do not have determinant line bundle, we can infer its curvature should be the trace of $\Rr$, of course, by trace we mean the $*$-product trace of the formal function $\left.\Rr\right|_{y=0}$ evaluated at tangent elements of $\E(M,\omega)$.

Recall that a \emph{trace} for a star product $\ast$ on a symplectic manifold $(M,\omega)$ is a character
$$ \tr : C_c^\infty(M)[[\nu]] \to \R[\nu^{-1},\nu]]$$
for the $*$-commutator, i.e. $ \tr([F,G]_\ast) = 0,$ for all $F,G\in C_c^\infty(M)[[\nu]]$.
There exists a unique normalised trace for a given star product on $(M,\omega)$. The normalisation is the following. Consider local equivalences $B$ of $*|_{C^\infty(U)[[\nu]]}$ with the Moyal star product $\ast_{\mathrm{Moyal}}$ on $U$ a contractible Darboux chart 
$B : (C^\infty(U)[[\nu]],\ast) \to (C^\infty(U)[[\nu]],\ast_{\mathrm{Moyal}})$
 so that $ BF\ast_{\mathrm{Moyal}} BG = B(F\ast G).$
We ask for the normalised trace to statisfy
\begin{equation*}\label{normalized_tr}
\tr(F) = \frac1{(2\pi\nu)^m} \int_M BF\ \frac{\omega^m}{m!}, \textrm{ for all } F\in C_c^\infty(U)[[\nu]].
\end{equation*}
The trace is written as an $L^2$-pairing with a formal function $\rho\in C^{\infty}(M)[\nu^{-1},\nu]]$, called the trace density.

We denote by \emph{$\tr^{*_{\nabla}}$ the normalised trace} of the Fedosov star product $*_{\nabla}$ and by \emph{$\rho^{\nabla}$ the trace density of  $\tr^{*_{\nabla}}$}.

\begin{theorem}\label{theor:omegatilde}
The formal $2$-form
\begin{equation*}
\widetilde{\Omega}^{\E}_{\nabla}(A,B):=(2\pi)^m.24.\nu^{m-2} \tr^{*_{\nabla}}(\left.\Rr_{\nabla}(A,B)\right|_{y=0}) \textrm{ for } A,B\in T_{\nabla}\E(M,\omega),
\end{equation*}
is a formal symplectic form on $\E(M,\omega)$ deforming $\Omega^{\E}$.\\
Moreover, the action of $\ham(M,\omega)$ on $\E(M,\omega)$ preserves $\widetilde{\Omega}^{\E}$.
\end{theorem}

\noindent We split the proof into several Lemmas.

\begin{lemma}[Bianchi identity]
For $A,B$ and $C$ be constant vector fields on $\E(M,\omega)$ then
$$\underset{A\,B\,C}{\stackrel{\curvearrowright}{\oplus}} \D_A (\left.\Rr_{\nabla} (B,C)\right|_{y=0})=0.$$
\end{lemma}

\begin{proof}
As $\Rr^{\nabla}(B,C)$ is a flat section, we use Lemma \ref{lemme:QDF} to get
$$Q^{\nabla}(\D_A  (\left.\Rr^{\nabla} (B,C)\right|_{y=0}))=\left( \ddto \Rr_{\nabla+tA}(B,C)+ \invnu[\alpha(A),\Rr^{\nabla}(B,C)]\right).$$
Writing $\Rr$ in terms of $\alpha$, we obtain
\begin{eqnarray*}
\underset{A\,B\,C}{\stackrel{\curvearrowright}{\oplus}} Q^{\nabla}(\D_A  (\left.\Rr^{\nabla} (B,C)\right|_{y=0})) & = & d\left(d\alpha+\invnu[\alpha(\cdot),\alpha(\cdot)]\right)(A,B,C) \\
& & +\underset{A\,B\,C}{\stackrel{\curvearrowright}{\oplus}}\invnu[\alpha(A),d\alpha(B,C)]
  + \underset{A\,B\,C}{\stackrel{\curvearrowright}{\oplus}}\frac{1}{\nu^2}[\alpha(A),[\alpha(B),\alpha(C)]].
\end{eqnarray*}
Because $d^2=0$ and Jacobi for the $\circ$-commutator, it simplifies to
\begin{eqnarray*}
\underset{A\,B\,C}{\stackrel{\curvearrowright}{\oplus}} Q^{\nabla}(\D_A  (\left.\Rr^{\nabla} (B,C)\right|_{y=0})) & = & \invnu d[\alpha(\cdot),\alpha(\cdot)](A,B,C) +\underset{A\,B\,C}{\stackrel{\curvearrowright}{\oplus}}\invnu[\alpha(A),d\alpha(B,C)].
\end{eqnarray*}
Direct computations show that the above RHS vanishes which concludes the proof.
\end{proof}

\noindent The next lemma recall the variation of the trace map from \cite{fed4}.

\begin{lemma} \label{lemme:tracevariation}
Let $t\mapsto \nabla^t$ be a smooth path of symplectic connections. Then
$$ \ddto \tr^{*_{\nabla^t}}(F)=\tr^{*_{\nabla}}\left(\left.\invnu[\alpha_{\nabla^0}(\ddto \nabla^t),Q^{\nabla}(F)]\right|_{y=0}\right) $$
\end{lemma}

\begin{proof}
The precise formula from \cite{fed4} is $$\ddto \tr^{*_{\nabla^t}}(F)=\tr^{*_{\nabla^0}}\left(\invnu[\left.(D^{\nabla})^{-1}\left( \overline{\ddto \nabla^t}+\ddto r^{\nabla^t}\right),Q^{\nabla}(F)]\right|_{y=0}\right).$$
By definition of $\alpha$, it gives Lemma \ref{lemme:tracevariation}
\end{proof}

\begin{lemma}\label{lemme:closedtildeomega}
The formal $2$-form $\widetilde{\Omega}^{\E}$ is closed.
\end{lemma}

\begin{proof}
It suffices to work with constant vector fields $A,B$ and $C$ on $\E(M,\omega)$ and to check $d\widetilde{\Omega}^{\E}(A,B,C)=0$, that is:
\begin{equation} \label{eq:cyclicsum}
\underset{A\,B\,C}{\stackrel{\curvearrowright}{\oplus}} A(\widetilde{\Omega}^{\E}(B,C))=0.
\end{equation}

Well, using the Lemma \ref{lemme:tracevariation}, we have (ommiting the constant multiple $(2\pi)^m.24.\nu^{m-2}$ for brevity) :
\begin{eqnarray*}
A(\widetilde{\Omega}^{\E}(B,C)) & = &A(\tr^{*_{\nabla}}(\Rr_{\nabla}(A,B)))\\
   & = &\tr^{*_{\nabla}}\left(\left.\invnu[\alpha_{\nabla}(A),Q^{\nabla}(F)]\right|_{y=0}\right)  
 + \tr^{*_{\nabla}}(\ddto \left.\Rr_{\nabla+tA}(B,C)\right|_{y=0})\\
  & = & \tr^{*_{\nabla}}(\D_A (\left.\Rr_{\nabla} (B,C)\right|_{y=0}))
\end{eqnarray*}
By Bianchi identity, the cyclic sum \eqref{eq:cyclicsum} vanishes which finishes the proof.
\end{proof}

\begin{lemma}\label{lemme:invtildeomega}
The formal $2$-form $\widetilde{\Omega}^{\E}$ is invariant under the natural action of $\ham(M,\omega)$ on $\E(M,\omega)$.
\end{lemma} 

Before we proceed to the proof, let us recall some known facts about the action of Hamiltonian diffeomorphisms on Weyl algebra sections. Let $\varphi\in \ham(M,\omega)$, it acts by pull-back on $\Gamma\W\otimes\Lambda(M)$:
$$\varphi^* (a\otimes \alpha):=\varphi^*a\otimes \varphi^*\alpha,$$
for $a\in \Gamma\W$ seen as a formal power series of symmetric tensors on $M$ and $\alpha\in \Omega(M)$ (it is actually an anti-action). All the ingredients of the Fedosov construction behaves well with respect to the pullback, in the sequel we will refer to the equations hereafter as the \emph{naturality} of Fedosov construction. For details we refer to the book \cite{fed}.
The pull-back preserves the $\circ$-product
$$\varphi^* (a\circ b)= \varphi^* a\circ \varphi^*b  \textrm{ for all } a,b \in \Gamma\W\otimes\Lambda(M).$$
The action on a symplectic connection $\nabla$ translates into a transformation of the covariant derivative $\partial^{\nabla}$:
$$\partial^{\varphi^{-1}\cdot \nabla} =\varphi^*\partial^{\nabla} (\varphi^{-1})^*.$$
We omit the symbol for composition as it can be misleading with the $\W$-algebra product. Similarly, we have
$$\delta=\varphi^*\delta(\varphi^{-1})^*.$$
The above leads to 
\begin{eqnarray*}
D^{\varphi^{-1}\cdot \nabla} & = & \varphi^*D(\varphi^{-1})^*\\
 & = & \partial^{\varphi^{-1}\cdot \nabla}-\delta + \invnu[\varphi^*r^{\nabla},\cdot],
\end{eqnarray*}
so that
$$r^{\varphi^{-1}\cdot \nabla}=\varphi^*r^{\nabla}$$
It means that if $a\in \Gamma\W_{D^{\nabla}}$ is a $D^{\nabla}$-flat section with symbol $a_0$, then $\varphi^*a\in \Gamma\W_{D^{\varphi^{-1}\cdot \nabla}}$ is a $D^{\varphi^{-1}\cdot \nabla}$-flat section with symbol $\varphi^*a_0$. So that, for all $F\in C^{\infty}(M)[[\nu]]$, we have
$$\varphi^*Q^{\nabla}(F)=Q^{\varphi^{-1}\cdot \nabla}(\varphi^*F).$$
The pull-back also induces an isomorphism of Fedosov star products
$$\varphi^*(F*_{\nabla}G)=(\varphi^*F)*_{\varphi^{-1}\cdot \nabla}(\varphi^*G).$$
Consequently, the trace density transforms as $\varphi^*\rho^{\nabla}=\rho^{\varphi^{-1}\cdot \nabla}$, so that
$$\tr^{*_{\varphi^{-1}\cdot \nabla}}(F)=\tr^{*_{\nabla}}((\varphi^{-1})^*F) \textrm{ for all } F\in C^{\infty}(M)[[\nu]].$$
All the above remains valid if $\varphi$ is a symplectic diffeomorphism, but this paper focuses on Hamiltonian diffeomorphisms.

\begin{lemma} \label{lemme:naturalalpha}
Let $\varphi$ be a Hamiltonian diffeomorphism, $\nabla$ a symplectic connection and $A\in T_{\nabla}\E(M,\omega)$, we have:
$$\varphi^*(\alpha_{\nabla}(A))=\alpha_{\varphi^{-1}\cdot \nabla}(\varphi^*A),$$
for $\varphi^*A$ the usual pullback on $\Gamma \mathrm{End(TM)}\otimes \Lambda^1M$.
\end{lemma}

\begin{proof}
Going back to Equation \eqref{eq:defalpha} defining $\alpha$ we have:
$$\alpha_{\varphi^{-1}\cdot \nabla}(\varphi^*A)=(D^{\varphi^{-1}\cdot \nabla})^{-1}(\varphi^*\overline{A}+\ddto r^{\varphi^{-1}\cdot  \nabla+t\varphi^*A}),$$
after observing that $\overline{\varphi^*A}=\varphi^*\overline{A}$. One also remark that $\ddto r^{\varphi^{-1}\cdot  \nabla+t\varphi^*A}=\varphi^*\ddto r^{ \nabla+tA}$, by naturality.
It means that $\alpha_{\varphi^{-1}\cdot \nabla}(\varphi^*A)$ is the unique solution to
\begin{displaymath}
\textrm{(i)}\ \left\{\begin{array}{ccl}
D^{\varphi^{-1}\cdot \nabla}\alpha_{\varphi^{-1}\cdot \nabla}(\varphi^*A) & = & \varphi^*(\overline{A}+\ddto r^{ \nabla+tA})\\
\left.\alpha_{\varphi^{-1}\cdot \nabla}(\varphi^*A)\right|_{y=0} & = & 0.
\end{array}\right.
\end{displaymath}
On the other hand,  $\alpha_{ \nabla}(A)$ is the unique solution to 
\begin{displaymath}
\left\{\begin{array}{ccl}
D^{\nabla}\alpha_{ \nabla}(A) & = & (\overline{A}+\ddto r^{ \nabla+tA})\\
\left.\alpha_{\nabla}(A)\right|_{y=0} & = & 0.
\end{array}\right.
\end{displaymath}
By naturality, $D^{\varphi^{-1}\cdot \nabla}= \varphi^*D(\varphi^{-1})^*$, also $\left.\varphi^*(\alpha_{\nabla}(A))\right|_{y=0} =0$, so that $\varphi^*(\alpha_{\nabla}(A))$ is also the solution of (i), which means
$$\varphi^*(\alpha_{\nabla}(A))=\alpha_{\varphi^{-1}\cdot \nabla}(\varphi^*(A)).$$
The equality is proved.
\end{proof}

\begin{proof}[Proof of Lemma \ref{lemme:invtildeomega}]
Let $\varphi$ be a Hamiltonian diffeomorphism, we want to prove that
$$\widetilde{\Omega}^{\E}_{\varphi^{-1}\cdot \nabla}(\varphi^*A,\varphi^*B) = \widetilde{\Omega}^{\E}_{\nabla}(A, B)$$
for $A,B\in T_{\nabla}\E(M,\omega)$ and the pullback $\varphi^*A= \ddto\varphi^{-1}\cdot( \nabla+tA)$ is equal to the differential of the action by $\varphi^{-1}$ on $\E(M,\omega)$. So we want to show
$$\tr^{*_{\varphi^{-1}\cdot \nabla}}\left(\left.\Rr_{\varphi^{-1}\cdot \nabla}(\varphi^*A,\varphi^*B)\right|_{y=0}\right) = \tr^{*_{\nabla}}\left(\left.\Rr_{\nabla}(A, B)\right|_{y=0}\right).$$

Let us have a look at the LHS. 
$$\left.\Rr_{\varphi^{-1}\cdot \nabla}(\varphi^*A,\varphi^*B)\right|_{y=0}=\left.\invnu[\alpha_{\varphi^{-1}\cdot \nabla}(\varphi^*A),\alpha_{\varphi^{-1}\cdot \nabla}(\varphi^*B)]\right|_{y=0}.$$
From Lemma \ref{lemme:naturalalpha}, we know
$$\varphi^*(\alpha_{\nabla}(A))=\alpha_{\varphi^{-1}\cdot \nabla}(\varphi^*A).$$
So that, by naturality
\begin{eqnarray*}
\tr^{*_{\varphi^{-1}\cdot \nabla}}\left(\left.\Rr_{\varphi^{-1}\cdot \nabla}(\varphi^*A,\varphi^*B)\right|_{y=0}\right) & = & \tr^{*_{\varphi^{-1}\cdot \nabla}}\left(\varphi^*\left.\invnu[\alpha_{ \nabla}(A),\alpha_{\nabla}((B))]\right|_{y=0}\right) \\
& = & \tr^{*_{ \nabla}}\left(\left.\invnu[\alpha_{ \nabla}(A),\alpha_{\nabla}((B))]\right|_{y=0}\right).
\end{eqnarray*}
It concludes the proof since $\left.\Rr_{\nabla}(A, B)\right|_{y=0}=\left.\invnu[\alpha_{ \nabla}(A),\alpha_{\nabla}((B))]\right|_{y=0}$.
\end{proof}

\begin{proof}[Proof of Theorem \ref{theor:omegatilde}]
The form $\widetilde{\Omega}^{\E}$ is closed by Lemma \ref{lemme:closedtildeomega}. To show it is a formal symplectic form, it is now enough to show it starts with $\Omega^{\E}$ at order zero in $\nu$. Well, since the trace $\tr^{*_{\nabla}}(F)=\frac{1}{(2\pi\nu)^m} \int_M F \dvol+O(\nu)$, from the computation of the first order term of $\left.\Rr_{\nabla}(\cdot,\cdot)\right|_{y=0}$ in Theorem \ref{theor:Ralpha}, we have
\begin{eqnarray*}
\widetilde{\Omega}^{\E}_{\nabla}(A,B) & = & \int_M \Lambda^{i_1j_1}\Lambda^{i_2j_2}\Lambda^{i_3j_3}\underline{A}_{i_1i_2i_3}\underline{B}_{j_1j_2j_3}\dvol+O(\nu) \\
 & = & \Omega^{\E}(A,B)+O(\nu).
\end{eqnarray*}
for $A,B\in T_{\nabla}\E(M,\omega)$.

The invariance with respect to the action of Hamiltonian diffeomorphisms is the content of Lemma \ref{lemme:invtildeomega}.
\end{proof}

\subsection{A formal moment map}

In \cite{cagutt}, a moment map for the action of $\ham(M,\omega)$ on $(\E(M,\omega),\Omega^{\E})$ is obtained and in \cite{LLF} this moment map is interpreted as the first term of the trace denisity of $*_{\nabla}$. We go one step further in this direction by showing the full star product trace $\tr^{*_{\nabla}}$ may be interpreted as a ``formal moment map'' for the action of $\ham(M,\omega)$ on $(\E(M\omega),\widetilde{\Omega}^{\E})$.

Denote by $C^{\infty}_0(M)$ the space of smooth sections with zero mean. The \emph{Cahen--Gutt moment map} is the map
$$\mu:\E(M,\omega)\rightarrow C^{\infty}(M):\nabla \mapsto \mu(\nabla)$$
defined by 
$$\mu(\nabla)=(\nabla^2_{p q} \textrm{Ric})^{pq}+ \frac{1}{4}R_{p_1p_2p_3p_4}R^{p_1p_2p_3p_4}-\frac{1}{2}\textrm{Ric}_{p_1p_2}\textrm{Ric}^{p_1p_2},$$
where $R$, resp. $\textrm{Ric}$, being the curvature, resp. Ricci curvature, of the given $\nabla$, we raise the indices with the matrix $\Lambda^{kl}$. The function $\mu(\nabla)$ can be seen as an element of the dual of $C^{\infty}_0(M)$ by using the $L^2$-pairing.

\begin{theorem}\label{theor:formalmomentmap}
The map 
$$\widetilde{\mu}: \E(M,\omega) \rightarrow L(C^{\infty}_0(M), \R[[\nu]]):\nabla\mapsto \left[\widetilde{\mu}(\nabla):H\mapsto -(2\pi)^m.24.\nu^{m-2}\tr^{*_{\nabla}}(H)\right]$$
is an equivariant formal moment map for the action of $\ham(M,\omega)$ on $(\E(M\omega),\widetilde{\Omega}^{\E})$, in the sense that
\begin{eqnarray}\label{eq:formalmomentmap}
\ddto \widetilde{\mu}(\nabla+tA)(H) & = &\widetilde{\Omega}^{\E}_{\nabla}(\Lr_{X_H}\nabla,A), \ \textrm{ (formal moment map),} \label{eq:formalmomentmap}\\
\widetilde{\mu}(\varphi\cdot\nabla)(H) & = & \widetilde{\mu}(\nabla)(\varphi^*H), \ \textrm{ (equivariance),} \label{eq:equivmomentmap}
\end{eqnarray}
for all $A \in T_{\nabla}\E(M,\omega)$, $\varphi\in \ham(M,\omega)$ and $H\in C^{\infty}_0(M)$.

Moreover, $\widetilde{\mu}$ is a formal deformation of $\mu(\nabla)$, that is
\begin{equation} \label{eq:deformmu}
\widetilde{\mu}(\nabla)(H)=\int_M H \mu(\nabla)\dvol + O(\nu),
\end{equation}
for all $H\in C^{\infty}_0(M)$.
\end{theorem}

\begin{rem}
Let us comment the smoothness of $\widetilde{\mu}$. Actually, $\widetilde{\mu}$ may be seen as a the map 
$$\E(M,\omega) \rightarrow C^{\infty}(M)[[\nu]]: \nabla \mapsto -(2\pi)^m.24.\nu^{m-2}\rho^{\nabla}.$$
In that picture, $\rho^{\nabla}$ depends, at each order in $\nu$, polynomially on the curvature of $\nabla$, and its covariant derivative which make it a smooth map from $\E(M,\omega)$ to $C^{\infty}(M)[[\nu]]$.
\end{rem}

\noindent Recall that we say that $*_{\nabla}$ is \emph{closed} if the trace density $\rho^{\nabla}$ is constant. We have the straightforward corollary.
\begin{cor}
A symplectic connection $\nabla$ gives a closed star product $*_{\nabla}$ if and only if $\widetilde{\mu}(\nabla)=0$.
\end{cor}

\noindent Again the proof of the above Theorem is splitted into several Lemmas. 

We first give formulas for the derivative of the action of $\ham(M,\omega)$ on $\Gamma\W \otimes \Omega(M)$. Consider a smooth map $t\in [0,1]\mapsto H_t\in C^{\infty}(M)$, it generates an Hamiltonian isotopy $\varphi_t^{H_{\cdot}}$ as the unique solution of the evolution equation:
\begin{equation*} \label{eq:Hamdiff}
\left\{\begin{array}{ccl}
\ddt \varphi_t^{H_{\cdot}}(\cdot) & = & X_{H_t,\varphi_t^{H_{\cdot}}(\cdot)},\\
\varphi_0^{H_{\cdot}}& = & Id,
\end{array}
\right.
\end{equation*}
where $ X_{H_t,\varphi_t^{H_{\cdot}}(\cdot)}$ is the time dependent vector field $X_{H_t}$ evaluated at the point $\varphi_t^{H_{\cdot}}(\cdot)$, in the indices of $\varphi_t^{H_{\cdot}}(\cdot)$, we wrote $H_{\cdot}$ to refer to the map $t\mapsto H_t$ on which the isotopy depends. Recall that the collection of all time-$1$ maps $\varphi_1^{H_{\cdot}}$ for all smooth maps $H_t$ gives the group of Hamiltonian diffeomorphisms.

\begin{lemma}\label{lemme:Liederiv}
Consider a smooth map $t\in [0,1]\mapsto H_t\in C^{\infty}(M)$, then the derivative of the action of $\varphi_t^{H_{\cdot}}$ on $\Gamma\W\otimes\Lambda(M)$ is given by the formula:
\begin{equation*}
\ddt (\varphi_t^{H_{\cdot}})^*=(\varphi_t^{H_{\cdot}})^*\left(\imath(X_{H_t})D^{\nabla}+D^{\nabla}\imath(X_{H_t})+\invnu\left[- \omega_{ij}y^i X_{H_t}^j + \frac{1}{2}(\nabla^2_{kq}H_t) y^k y^q-\imath(X_{H_t})r^{\nabla},\cdot \right]\right)
\end{equation*}
\end{lemma}

\begin{proof}
One first notice that $\ddt (\varphi_t^{H_{\cdot}})^*a=(\varphi_t^{H_{\cdot}})^*(\Lr_{X_{H_t}}a)$, for $a\in \Gamma\W\otimes\Lambda M$ and $\Lr$ denoting the Lie derivative of tensors. Then, we use the formula for the Lie derivative from \cite{gr3}.
\end{proof}

\begin{lemma}\label{lemme:QHformula}
Let $H\in C^{\infty}(M)$ and $\nabla\in \E(M,\omega)$, we have:
$$Q^{\nabla}(H)=H-\omega_{ij}y^i X_{H}^j + \frac{1}{2}(\nabla^2_{kq}H) y^k y^q-\imath(X_{H})r^{\nabla}+ \alpha_{\nabla}(\Lr_{X_H}\nabla).$$
\end{lemma}

\begin{proof}
It is enough to check the RHS is a flat section, as at $y=0$ both sides equal $H$.

Let us start with $\alpha_{\nabla}(\Lr_{X_H}\nabla)$.
$$D^{\nabla}\alpha_{\nabla}(\Lr_{X_H}\nabla)=\overline{\Lr_{X_H}\nabla}+ \ddto r^{(\varphi_t^{H})^{-1}\cdot\nabla}.$$
By naturality, $ \ddto r^{(\varphi_t^{H})^{-1}\cdot\nabla}= \ddto (\varphi_t^{H})^{*}r^{\nabla}$. Then we use Lemma \ref{lemme:Liederiv} to obtain
$$\ddto r^{(\varphi_t^{H})^{-1}\cdot\nabla}=\imath(X_{H})D^{\nabla}r^{\nabla}+D^{\nabla}\imath(X_{H})r^{\nabla}+\invnu\left[- \omega_{ij}y^i X_{H}^j + \frac{1}{2}(\nabla^2_{kq}H) y^k y^q-\imath(X_{H})r^{\nabla},r^{\nabla} \right].$$
From Equation \eqref{eq:req} defining $r^{\nabla}$, we have $D^{\nabla}r^{\nabla}=\frac{1}{2\nu}[r^{\nabla},r^{\nabla}]-\overline{R}$, so that
$$\ddto r^{(\varphi_t^{H})^{-1}\cdot\nabla}=D^{\nabla}\imath(X_{H})r^{\nabla}+\invnu\left[- \omega_{ij}y^i X_{H}^j + \frac{1}{2}(\nabla^2_{kq}H) y^k y^q,r^{\nabla}\right]-\imath(X_H)\overline{R}.$$
Now, since $(\Lr_{X_H}\nabla)_{ij}^k=(\nabla^2_{ij}X_H)^k+(R(X_H,\cdot)\cdot)_{ij}^k$, we end with
\begin{equation}\label{eq:Dalpha}
D^{\nabla}\alpha_{\nabla}(\Lr_{X_H}\nabla)=D^{\nabla}\imath(X_{H})r^{\nabla}+\invnu\left[- \omega_{ij}y^i X_{H}^j + \frac{1}{2}(\nabla^2_{kq}H) y^k y^q,r^{\nabla} \right]+\frac{1}{2}\omega_{lk}(\nabla^2_{ij}X_H)^ky^ly^k dx^l.
\end{equation}

Now, it remains to compute $D^{\nabla}(H-\omega_{ij}y^i X_{H}^j + \frac{1}{2}(\nabla^2_{kq}H) y^k y^q-\imath(X_{H})r^{\nabla})$. Let us do it term by term.
First,
\begin{equation}\label{eq:DH}
D^{\nabla}H=dH.
\end{equation}
After that, 
\begin{equation}\label{eq:Domega}
D^{\nabla}\omega_{ij}y^i X_{H}^j=-\nabla^2_{pq}Hy^qdx^p+dH+\invnu[r^{\nabla},\omega_{ij}y^i X_{H}^j].
\end{equation}
The next one gives
\begin{equation}\label{eq:Dnabla2}
D^{\nabla}(\frac{1}{2}(\nabla^2_{kq}H) y^k y^q)=-\frac{1}{2}\omega_{pr}(\nabla^2_{ij}X_H)^ry^py^jdx^i-\nabla^2_{pq}Hy^qdx^p+\frac{1}{2\nu}[r^{\nabla},(\nabla^2_{kq}H) y^k y^q],
\end{equation}
where in the RHS's of the last three equations each term corresponds to the $\partial^{\nabla}$, $\delta$ and $\invnu[r^{\nabla},\cdot]$ terms, in this order. We do not compute in details the term $D^{\nabla}\imath(X_{H})r^{\nabla}$.

Combining Equations \eqref{eq:Dnabla2}, \eqref{eq:Domega}, \eqref{eq:DH} and \eqref{eq:Dalpha} we obtain 
$$D^{\nabla}(H-\omega_{ij}y^i X_{H}^j + \frac{1}{2}(\nabla^2_{kq}H) y^k y^q-\imath(X_{H})r^{\nabla}+ \alpha_{\nabla}(\Lr_{X_H}\nabla))=0,$$
which shows the formula stated in the Lemma.
\end{proof}

\begin{proof}[Proof of Theorem \ref{theor:formalmomentmap}]
We have to prove the formal moment map Equation \eqref{eq:formalmomentmap}. From the variation formula from Lemma \ref{lemme:tracevariation}, we have
$$\ddto \tr^{\nabla+tA}(H)=\tr^{\nabla}(\left.\invnu[\alpha_{\nabla}(A),Q(H)]\right|_{y=0}).$$
Also,
\begin{eqnarray*}
\widetilde{\Omega}^{\E}_{\nabla}(\Lr_{X_H}\nabla,A) & = &(2\pi)^m.24.\nu^{m-2} \tr^{\nabla}(\left.\Rr^{\nabla}(\Lr_{X_H}\nabla,A)\right|_{y=0})\\
 & = & (2\pi)^m.24.\nu^{m-2} \tr^{\nabla}(\left.\invnu[\alpha_{\nabla}(\Lr_{X_H}\nabla),\alpha_{\nabla}(A)]\right|_{y=0}).
\end{eqnarray*}

So that to show the moment map Equation \eqref{eq:formalmomentmap}, we need to prove 
\begin{equation}\label{eq:toshow}
\tr^{\nabla}(\left.\invnu[\alpha_{\nabla}(A),Q(H)-\alpha_{\nabla}(\Lr_{X_H}\nabla)]\right|_{y=0})=0.
\end{equation}
Now, using Lemma \ref{lemme:QHformula}, we obtain ($H$ being central)
$$\invnu[\alpha_{\nabla}(A),Q(H)-\alpha_{\nabla}(\Lr_{X_H}\nabla)]= \invnu[\alpha_{\nabla}(A),-\omega_{ij}y^i X_{H}^j + \frac{1}{2}(\nabla^2_{kq}H) y^k y^q-\imath(X_{H})r^{\nabla}].$$
By Lemma \ref{lemme:Liederiv}, we get
$$\invnu[\alpha_{\nabla}(A),Q(H)-\alpha_{\nabla}(\Lr_{X_H}\nabla)]=-\ddto (\varphi^H_t)^*\alpha_{\nabla}(A) + \imath(X_{H})D^{\nabla}\alpha_{\nabla}(A)+D^{\nabla}\imath(X_{H})\alpha_{\nabla}(A).$$
Finally, at $y=0$, $\alpha_{\nabla}(A)$ vanishes and so does $\ddto (\varphi^H_t)^*\alpha_{\nabla}(A)$. Also, $D^{\nabla}\alpha_{\nabla}(A)$ contains at least a $y$ in every term and $\imath(X_{H})\alpha_{\nabla}(A)=0$, since $\alpha_{\nabla}(A)$ is a $0$-form. It shows Equation \eqref{eq:toshow} is satisfied, proving the formal moment map Equation.

The equivariance Equation \eqref{eq:equivmomentmap} is a direct consequence of the naturality of Fedosov construction.

Equation \eqref{eq:deformmu} is direct when one notices that at order $0$ in $\nu$ the formal moment map equation becomes the moment map equation for the action of $\ham(M,\omega)$ on $(\E(M,\omega),\Omega^{\E})$ as in \cite{cagutt}.
\end{proof}

\section{Applications: Hamiltonian diffeomorphisms and automorphisms of star products} \label{sect:Applications}

We first recall the construction of Hamiltonian automorphisms from \cite{LLF0}, by solving Heisenberg equation at the level of a quantum algebra $(C^{\infty}(M)[[\nu]],*)$ on $(M,\omega)$. Recall that a derivation $D$ of $*$ is \emph{quasi-inner} if there is $H\in C^{\infty}(M)[[\nu]]$ so that $D=D_H:=\invnu[H,\cdot]_*$. For a smooth map $t\in [0,1]\mapsto H_t \in C^{\infty}(M)[[\nu]]$ there exists a unique $A_t^{H_{\cdot}}$ \emph{path of Hamiltonian automorphisms} of the star product $*$ satisfying the Heisenberg equation
\begin{equation*} \label{eq:Hamaut}
\left\{\begin{array}{ccl}
\ddt A_t^{H_{\cdot}} & = & D_{H_t}A_t^{H_{\cdot}},\\
A_0^{H_{\cdot}}& = & Id,
\end{array}
\right.
\end{equation*}
where we use similar notations as for Hamiltonian isotopies. Moreover, if we write $H_t=H_{0,t}+\nu \ldots$ then $A_t^{H_{\cdot}} =\left((\varphi_t^{-H_{0,\cdot}} )^{-1}\right)^*+\nu\ldots.$

The \emph{group of Hamiltonian automorphisms} \cite{LLF0} of a star product $*$ is defined by
$$\ham(M,*):=\{A=A_1^{H_{\cdot}} \textrm{ for smooth } t\mapsto H_t\in C^{\infty}(M)[[\nu]]\}.$$
It projects naturally to $\ham(M,\omega)$ by $\cl:\ham(M,*)\rightarrow \ham(M,\omega):A=\varphi^*+\nu\ldots\mapsto \varphi$.

In \cite{LLF0}, the question of determining $\cl^{-1}(Id)$ is raised. It amounts to understand Hamiltonian automorphisms of the form $A_1^{H_{\cdot}}= Id+\nu \ldots$ for $-H_{0,t}$ generating a loop of Hamiltonian diffeomorphisms $\{(\varphi_t^{-H_{0,\cdot}} )^{-1}\}$ . Any automorphism of $*$ of the form $A=Id+\nu \ldots$ is of the form $\exp(\nu D)$ for a derivation $D$ of the star product. Moreover, all automorphisms $\exp(\nu D_G)$  for $G\in  C^{\infty}(M)[[\nu]]$ are in $\ham(M,*)$. When $H^1_{dR}(M)\neq 0$ not all derivations are quasi-inner, it is then a natural question to wonder: what does the quotient 
$$\cl^{-1}(Id)/\{\exp (\nu D_{G}) \textrm{ for } G\in  C^{\infty}(M)[[\nu]]\}$$
looks like? In \cite{LLF0}, it is shown this quotient depends only on the $\pi_1(\ham(M,\omega))$ and the characteristic class of $*$. Also, if  the above quotient is, say, at most countable, it implies the local model of $\ham(M,*)$ is an open subset af $C^{\infty}(M)[[\nu]]$.

\subsection{Parallel transport and Hamiltonian automorphisms}

Let $H_t \in C^{\infty}(M)$ generating $\varphi_t^{H_{\cdot}}$. Consider the smooth path of connection $\{(\varphi_t^{H_{\cdot}})^{-1}\cdot\nabla\}$, for a given symplectic connection $\nabla$. Recall that by naturality of the Fedosov construction $D^{(\varphi_t^{H_{\cdot}})^{-1}\cdot\nabla}=(\varphi_t^{H_{\cdot}})^{*}D^{\nabla}((\varphi_t^{H_{\cdot}})^{-1})^*$ and that $(\varphi_t^{H_{\cdot}})^{*}$ is an isomorphism of flat sections algebra:
$$(\varphi_t^{H_{\cdot}})^{*}:\Gamma\W_{D^{\nabla}} \stackrel{\cong}{\rightarrow} \Gamma\W_{D^{(\varphi_t^{H_{\cdot}})^{-1}\cdot\nabla}}.$$

Now, we consider $v_t\in\Gamma\W^+$ generated by $h_t:=-\alpha_{(\varphi_t^{H_{\cdot}})^{-1}\cdot\nabla}(\frac{d}{dt}(\varphi_t^{H_{\cdot}})^{-1}\cdot\nabla)$ obtained from Theorem \ref{theor:smoothisom}. The conjugation $a \in \Gamma\W_{D^{\nabla}}\mapsto v_t\circ a\circ v_t^{-1}\in \Gamma\W_{D^{(\varphi_t^{H_{\cdot}})^{-1}\cdot\nabla}}$  gives the parallel transport for the formal connection $\D$ along $\{(\varphi_t^{H_{\cdot}})^{-1}\cdot\nabla\}$.

We define the automorphism $B_t$ of $(C^{\infty}(M)[[\nu]],*_{\nabla})$ by
$$B_t:C^{\infty}(M)[[\nu]]\rightarrow C^{\infty}(M)[[\nu]] :\left.((\varphi_t^{H_{\cdot}})^{-1})^*(v_t\circ Q^{\nabla}(F) \circ v_t^{-1})\right|_{y=0}. $$
Inspired by \cite{FU} and \cite{Ch2}, we give the link between parallel transport and solution to Heisenberg equation.

\begin{prop} \label{prop:Btparallel}
For all $F\in C^{\infty}(M)[[\nu]]$ we have
\begin{equation*} \label{eq:Hamparallel}
\left\{\begin{array}{ccl}
\ddt B_t(F) & = &- \invnu[H_t,B_t(F)]_{*_{\nabla}},\\
B_0(F)& = & F.
\end{array}
\right.
\end{equation*}
Hence, $B_t$ is a Hamiltonian automorphism of $*_{\nabla}$.
\end{prop}

\begin{proof}
Clearly, $B_0$ is the identity. We write $\widetilde{B}_t(F)=((\varphi_t^{H_{\cdot}})^{-1})^*(v_t\circ Q^{\nabla}(F) \circ v_t^{-1})$ and we compute 
\begin{eqnarray*}
\ddt  \widetilde{B}_t(F) & =  & -((\varphi_t^{H_{\cdot}})^{-1})^*(\ddt(\varphi_t^{H_{\cdot}})^*)\widetilde{B}_t(F)+((\varphi_t^{H_{\cdot}})^{-1})^*\ddt v_t\circ Q^{\nabla}(F) \circ v_t^{-1} \\
 & = & -\invnu[- \omega_{ij}y^i X_{H_t}^j + \frac{1}{2}(\nabla^2_{kq}H_t) y^k y^q-\imath(X_{H_t})r^{\nabla},\widetilde{B}_t(F)]\\
  & & -\invnu[((\varphi_t^{H_{\cdot}})^{-1})^*\alpha_{(\varphi_t^{H_{\cdot}})^{-1}\cdot\nabla}(\frac{d}{dt}   (\varphi_t^{H_{\cdot}})^{-1}\cdot\nabla),\widetilde{B}_t(F)]\\
  & = & -\invnu[- \omega_{ij}y^i X_{H_t}^j + \frac{1}{2}(\nabla^2_{kq}H_t) y^k y^q-\imath(X_{H_t})r^{\nabla},\widetilde{B}_t(F)]-\invnu[\alpha_{\nabla}(\Lr_{X_{H_t}}\nabla),\widetilde{B}_t(F)] \\ 
  & = & -\invnu[Q^{\nabla}(H_t),\widetilde{B}_t(F)],
\end{eqnarray*}
where we used Lemma \ref{lemme:Liederiv} and the definition of $v_t$ to get the second equality, Lemma \ref{lemme:naturalalpha} for the third one and Lemma \ref{lemme:QHformula} for the last equality.
Equating both sides at $y=0$, we obtain that $B_t$ is indeed the path of Hamiltonian automorphism generated by $-H_t$.
\end{proof}

Now, assume $H_t \in C^{\infty}(M)$ generates a loop of Hamiltonian diffeomorphisms $\{\varphi_t\}$. Then, $B_1(F)=\left.v_1\circ Q^{\nabla}(F) \circ v_1^{-1}\right|_{y=0}=F+\nu\ldots$. We will now show that $v_1$ is generated by curvature element and deduce from it that $B_1=\exp(\nu D_G)$ for some $G\in C^{\infty}(M)[[\nu]]$. 

More generally, we study the holonomy of $\D$ and obtain hereafter an analogue to Ambrose--Singer Theorem \cite{AS} for the formal connection $\D$.

\begin{lemma}\label{lemme:AStheorem}
Let $t\in [0,1] \mapsto \nabla^t\in \E(M,\omega)$ a loop at $\nabla^0:=\nabla$ and consider it as the boundary of a disk
$$[0,1]^2 \rightarrow \E(M,\omega): (t,s)\mapsto \nabla^{ts},$$
with $\nabla^{t1}=\nabla^t$ and $\nabla^{t0}=\nabla^{0s}=\nabla^{1s}=\nabla$ for all $t$ and $s$.
Let $t\mapsto w_{ts}\in \Gamma \W^+$ be generated by $-\alpha_{\nabla^{ts}}(\ddt \nabla^{ts})$. Then, $s\mapsto w_{1s}$ is generated by the flat section 
\begin{equation}\label{eq:holonomie}
w_{1s}\circ\int_0^1 w_{ts}^{-1}\circ \Rr_{\nabla^{ts}}(\ddt\nabla^{ts},\dds \nabla^{ts})\circ w_{ts}\,dt \circ w_{1s}^{-1}.
\end{equation}
\end{lemma}

\begin{proof}
The candidate generator of $s\mapsto w_{1s}$ is a flat section. Indeed, from Theorem \ref{theor:smoothisom}, we know that conjugation with $w_{1s}$ preserves $D^{\nabla}$-flat sections. So, it remains to prove the integrand of \eqref{eq:holonomie} is $D^{\nabla}$-flat. Again using Theorem \ref{theor:smoothisom}, we have:
\begin{equation*}
D^{\nabla}(w_{ts}^{-1}\circ \Rr_{\nabla^{ts}}(\ddt\nabla^{ts},\dds \nabla^{ts})\circ w_{ts})=w_{ts}^{-1}\circ (D^{\nabla^{ts}} \Rr_{\nabla^{ts}}(\ddt\nabla^{ts},\dds \nabla^{ts}))\circ w_{ts}=0,
\end{equation*}
because the curvature $\Rr$ of $\D$ takes values in flat sections.

We now show expression \eqref{eq:holonomie} is equal to $\nu\,(\dds w_{1s})\circ w_{1s}^{-1}$. Actually, it is easier to treat $w_{1s}^{-1}\circ \dds w_{1s}$. The computation is rather standard and comes mainly from \cite{KM} and \cite{Magnot}.
\begin{eqnarray*}
w_{1s}^{-1}\circ \dds w_{1s} & = & \int_0^1 \ddt(w_{ts}^{-1}\circ \dds w_{ts}) dt \\
 & = & -\invnu \int_0^1 w_{ts}^{-1} \circ\dds \alpha_{\nabla^{ts}}(\ddt \nabla^{ts}) \circ w_{ts} dt.
\end{eqnarray*}
Using the formula \eqref{eq:Ralpha} for $\Rr_{\nabla^{ts}}$, we have
\begin{eqnarray*}
w_{1s}^{-1}\circ \dds w_{1s} & =  & -\invnu\int_0^1  w_{ts}^{-1} \circ\left(\ddt \alpha_{\nabla^{ts}}(\dds \nabla^{ts})+\invnu[ \alpha_{\nabla^{ts}}(\ddt \nabla^{ts}),\alpha_{\nabla^{ts}}(\dds \nabla^{ts}) ]\right) \circ w_{ts} \, dt \\
 &  & +  \invnu \int_0^1  w_{ts}^{-1} \circ \Rr_{\nabla^{ts}} (\ddt \nabla^{ts},\dds \nabla^{ts})  \circ w_{ts} \, dt.
\end{eqnarray*}
Now, the first line of the above equation vanishes by integration by parts:
\begin{eqnarray*}
\int_0^1  w_{ts}^{-1} \circ\left(\ddt \alpha_{\nabla^{ts}}(\dds \nabla^{ts})\right) \circ w_{ts} \, dt  & = &  \left[w_{ts}^{-1} \circ \alpha_{\nabla^{ts}}(\dds \nabla^{ts}) \circ w_{ts}\right]^{t=1}_{t=0} \\
 & & - \int_0^1  w_{ts}^{-1} \circ\invnu[ \alpha_{\nabla^{ts}}(\ddt \nabla^{ts}),\alpha_{\nabla^{ts}}(\dds \nabla^{ts}) ] \circ w_{ts} \, dt 
\end{eqnarray*}
with first term of the RHS vanishing because $\nabla^{0s}=\nabla^{1s}=\nabla$ for all $s$.
\end{proof}

\begin{theorem} \label{theor:holohamilt}
Let $H_t$ generating a loop $\{\varphi_t^{H_{\cdot}}\} \in \ham(M,\omega)$ and let $B_t$ be the Hamiltonian automorphisms from Proposition \ref{prop:Btparallel}. Then, 
$$B_1=\exp(\nu D_G),$$
with $G$ explicitely determined by parallel transport and the curvature of $\D$ on a disk inside $\E(M,\omega)$ with boundary $\{(\varphi_t^{H_{\cdot}})^{-1}\cdot \nabla\}$.
\end{theorem}

\begin{proof}
As $\E(M,\omega)$ is an affine space we consider for example the disk
$$(t,s)\in [0,1]^2 \mapsto \nabla^{ts}:=\nabla + s((\varphi_t^{H_{\cdot}})^{-1}\cdot \nabla-\nabla),$$
and apply the preceding Lemma. To make notations coherent with the construction of $B_t$, we set $v_{ts} \in \Gamma \W^+$ generated by $-\alpha_{\nabla^{ts}}(\ddt \nabla^{ts})$, so that
$$B_1(F) = \left.(v_{11}\circ Q^{\nabla}(F)\circ v_{11}^{-1})\right|_{y=0}.$$
From the preceding Lemma, $s\mapsto v_{1s}$ is generated by the flat section
$$g_{s}:=v_{1s}\circ\int_0^1 v_{ts}^{-1}\circ \Rr_{\nabla^{ts}}(\ddt\nabla^{ts},\dds \nabla^{ts})\circ v_{ts}\,dt \circ v_{1s}^{-1}.$$
Then, the path of automorphism $s\mapsto B_1^s$ defined by $B_1^s(F) = \left.(v_{1s}\circ Q^{\nabla}(F)\circ v_{1s}^{-1})\right|_{y=0}$ is Hamiltonian and satisfies
\begin{equation*} \label{eq:Hamparallel}
\left\{\begin{array}{ccl}
\dds B_1^s(F) & = & \invnu[G_s,B_1^s(F)]_{*_{\nabla}},\\
B_1^0(F)& = & F.
\end{array}
\right.
\end{equation*}
for $G_s:= \left.g_s\right|_{y=0}\in \nu^2 C^{\infty}(M)[[\nu]]$.

Finally, $B_1$ can be made into a genuine exponential. Indeed, writing $G_s=\nu^2G_{s,2}+O(\nu^3)$, then $B_1^s=Id +  \invnu[\int_0^s G_{s,2}ds, F] + O(\nu^3)=\exp( \invnu D_{\int_0^s G_{s,2}ds})+O(\nu^3)$. Continuing this process by induction on the $\nu$-degree, we obtain $B_1^s$ as a product of exponentials and one is able to write $B_1=\exp(\nu D_{\widetilde{G}})$, as a single exponential, for some $\widetilde{G}$ depending on $G_s$ through the Campbell-Baker-Haussdorff formula.
\end{proof}

\begin{rem}
For $\Omega:=\nu \Omega_1+\nu^2\ldots$ with $\Omega_1\neq 0$, consider $*_{\nabla,\Omega}$ and $B_t$ the Hamiltonian automorphisms generate by $-H_t$ as in Proposition \ref{prop:Btparallel}. The exactness of the form $\int_0^1 (\varphi_t^{H_{\cdot}})^*\imath(X_{H_t})\Omega_1dt$ is an obstruction to write $B_1$ as an exponential, see \cite{LLF0}.
\end{rem}

\subsection{A formal action homomorphism} \label{sect:formalaction}

Shelukhin \cite{Shel} adapted the Weinstein action homomorphism \cite{Wein} in the context of an infinite dimensional symplectic manifold on which there is an action of a Lie group, possibly infinite dimensional, with moment map. This section is a first step in the generalisation of the work of Charles \cite{Ch} to our formal settings, see the Discussion Subsection for more details.

Let us describe the Weinstein action homomorphism for the particular symplectic space $(\E(M,\omega),\Omega^{\E})$ on which the group $\ham(M,\omega)$ with moment map $\mu$. Note that the moment map equation and the equivariance are simply the $(\nu=0)$-term of the corresponding Equations of Theorem \ref{theor:formalmomentmap}. Consider a loop $\{\varphi_t\}$ in $\ham(M,\omega)$ generated by $H_t\in C^{\infty}_0(M)$, and the corresponding loop $\{\varphi_t\cdot \nabla\}$ for $\nabla\in \E(M,\omega)$. Since $\E(M,\omega)$ is an affine space, one can consider a disk $B$ with boundary $\partial B= \{\varphi_t\cdot \nabla\}$. Shelukhin showed that the map 
$$\mathcal{A}_{\nabla}:\pi_1(\ham(M,\omega)) \rightarrow \R: \{\varphi_t\} \mapsto \A(\{\varphi_t\}):= \int_B \Omega^{\E} - \int_0^1 \int_M H_t\,\mu(\varphi_t\cdot\nabla)\dvol\, dt$$
is a well-defined homomorphism of groups on $\pi_1(\ham(M,\omega))$, independent of $\nabla$.

Deforming this construction with our formal moment map picture in mind, we have the following Theorem.

\begin{theorem} \label{theor:formalshelukhin}
With the above notation, there is a well-defined homomorphism
$$\widetilde{\A}_{\nabla}: \pi_1(\ham(M,\omega)) \rightarrow \R[[\nu]]: \{\varphi_t\} \rightarrow \widetilde{\A}_{\nabla}(\{\varphi_t\}):=  \int_B \widetilde{\Omega}^{\E} - \int_0^1 \left[\widetilde{\mu}(\varphi_t\cdot \nabla)\right](H_t)\dvol\, dt,$$
independent of $\nabla$
\end{theorem}

\begin{proof}
The proof follows directly from Shelukhin's proof in \cite{Shel}. The only difference comes from the fact that we need to use the fundamental Theorem of differential calculus for smooth maps from an interval into $\R[[\nu]]$ and Stokes Theorem for integration on spheres of the $2$-form $\widetilde{\Omega}^{\E}$ which takes values in $\R[[\nu]]$ instead of $\R$. Since $\R[[\nu]]$ is Fr\'echet, these Theorems are valid as one can see from \cite{KHN}.
\end{proof}

\noindent From Lemma \ref{lemme:AStheorem}, we interpret the integral of $\widetilde{\Omega}^{\E}$ in terms of the holonomy of $\D$.  

\begin{prop} \label{prop:Atildeholonomie}
Let $\{\varphi_t^{H_{\cdot}}\}$ be a loop in $\ham(M,\omega)$ and consider a disk $B$ parametrized by $[0,1]^2\rightarrow \E(M,\omega):(t,s)\mapsto \nabla^{ts}$ with boundary $\nabla^{t1}=\varphi_t^{H_{\cdot}}\cdot \nabla$ and $\nabla^{t0}=\nabla^{0s}=\nabla^{1s}=\nabla$ for all $t$ and $s$.
Let $t\mapsto v_{ts}\in \Gamma \W^+$ be generated by $-\alpha_{\nabla^{ts}}(\ddt \nabla^{ts})$. Then, 
\begin{equation}\label{eq:Atildeholonomie}
\widetilde{\A}_{\nabla}(\{\varphi_t\})=(2\pi)^m.24.\nu^{m-2}\left(\int_0^1 \tr^{*_{\nabla}}(\nu(\dds v_{1s})\circ v_{1s}^{-1})ds + \int_0^1 \tr^{*_{\nabla}}((\varphi_t^{H_{\cdot}})^*H_t) dt\right).
\end{equation}
\end{prop}

\begin{proof}
By Equation \eqref{eq:holonomie}, we have 
$$\tr^{*_{\nabla}}(\nu\left.(\dds v_{1s})\circ v_{1s}^{-1}\right|_{y=0})=\tr^{*_{\nabla}}(\left.v_{1s}\circ\int_0^1 v_{ts}^{-1}\circ R^{\nabla^{ts}}(\ddt\nabla^{ts},\dds \nabla^{ts})\circ v_{ts}\,dt \circ v_{1s}^{-1}\right|_{y=0}).$$
Because we used normalised trace and conjugation with $v_{1s}$ induces an automorphism of $*_{\nabla}$, we have
$$\tr^{*_{\nabla}}(\nu\left.(\dds v_{1s})\circ v_{1s}^{-1}\right|_{y=0})= \tr^{*_{\nabla}}(\left.\int_0^1 v_{ts}^{-1}\circ R^{\nabla^{ts}}(\ddt\nabla^{ts},\dds \nabla^{ts})\circ v_{ts}\right|_{y=0}\,dt )$$
The trace being linear, it goes under the integral symbol, and since conjugation with $v_{ts}$ induces an isomorphism between $*_{\nabla}$ and $*_{\nabla^{ts}}$ we conclude
\begin{eqnarray*}
\tr^{*_{\nabla}}(\left.\nu(\dds v_{1s})\circ v_{1s}^{-1}\right|_{y=0}) & = & \int_0^1 \tr^{*_{\nabla^{ts}}}(\left. R^{\nabla^{ts}}(\ddt\nabla^{ts},\dds \nabla^{ts})\right|_{y=0})\,dt, \\
 & = & \frac{1}{(2\pi)^m.24.\nu^{m-2}}  \int_0^1\widetilde{\Omega}^{\E}_{\nabla^{ts}}(\ddt\nabla^{ts},\dds \nabla^{ts})\, dt.
\end{eqnarray*}
So that
$$(2\pi)^m.24.\nu^{m-2}\int_0^1 \tr^{*_{\nabla}}(\nu(\dds v_{1s})\circ v_{1s}^{-1})ds= \int_B\widetilde{\Omega}^{\E}.$$
The second term of $\widetilde{\A}_{\nabla}$ is obtain directly from the definition of $\tilde{\mu}$ and the naturality of the trace functional.
\end{proof}

\begin{rem}
Comparing the expression \eqref{eq:Atildeholonomie} with Theorem \ref{theor:holohamilt}, we see that $-(\varphi^{H_{\cdot}}_t)^*H_t$ is the generator of the Hamiltonian automorphisms path obtained by parallel transport along $\{\varphi_t^{H_{\cdot}}\cdot\nabla\}$. Also, $\nu(\dds v_{1s})\circ v_{1s}^{-1}$ is the generator of a path of Hamiltonian automorphism of the form $Id+\nu\ldots$ with same endpoint as the previous one. Concatenating the two paths, we get a loop of Hamiltonian automorphisms and the above formula states that $\widetilde{\A}$ is simply the integral of the trace of the generator of that loop of Hamiltonian automorphisms.\\
In \cite{Ch}, within the framework of Berezin-Toeplitz quantization, the parallel transport maps are realized as unitary transformations. The corresponding formula obtained by Charles is a lift of the actual determinant map.
\end{rem}

Inspired by Shelukhin \cite{Shel0,Shel}, we relate the formal action functional $\widetilde{\A}$ to the invariant from \cite{FutLLF2} obstructing the closedness of $*_{\nabla}$, on K\"ahler manifolds. On $(M,\omega,J)$ a K\"ahler manifolds, let us consider loops $\{\varphi_t^{H_{\cdot}}\}$ in $K:=\ham(M,\omega)\cap \textrm{Hol}(M,J)$ the group of Hamiltonian isometries. It means each $\varphi_t$ preserves the complex structure. Consider the Levi-Civita connection $\nabla$ of the K\"ahler manifold, the corresponding loop of symplectic connections is trivial $\{\varphi_t^{H_{\cdot}}\cdot\nabla\}=\{\nabla\}$. To compute $\widetilde{\A}_{\nabla}(\{\varphi_t^{H_{\cdot}}\})$, we can take the trivial disk consisting of the single point $\{\nabla\}$, such that there is no contributions of $\widetilde{\Omega}^{\E}$:
$$\widetilde{\A}_{\nabla}(\{\varphi_t^{H_{\cdot}}\})=(2\pi)^m.24.\nu^{m-2}\int_0^1 \tr^{*_{\nabla}}(H_t)dt.$$
Writing $\tau_{\cdot}$ for the moment map of the action of $K$ on $M$ normalised to have zero integral, in \cite{FutLLF2}, Futaki and the author showed that the map 
$$\F:Lie(K) \rightarrow \R[[\nu]]:X \mapsto \tr^{*_{\nabla}}(\tau_X)$$
is a Lie algebra character, that is independent of the choice of a $K$-invariant symplectic connection $\nabla$, that obstructs the closedness of $*_{\nabla}$. In \cite{LLF2}, we observed that the first order term of $\F$ computed on the Ono--Sano--Yotsutani \cite{Ono} example is non-trivial. Using the isomorphism $$\pi_1(K)\otimes \R\cong Lie(K)/[Lie(K),Lie(K)],$$ we obtain:

\begin{cor} \label{cor:futaki}
The invariant $\left.\widetilde{\A}_{\nabla}\right|_{\pi_1(K)}$ and $\F$ coincides up to a constant factor in $\R[[\nu]]$. So, $\widetilde{\A}_{\nabla}$ itself is not trivial in general.
\end{cor}

\subsection{Discussion}

\begin{enumerate}
\item The formal moment map Theorem \ref{theor:formalmomentmap}, is a formal counter-part of the results of Foth--Uribe \cite{FU}. Using Berezin-Toeplitz operators, they obtain after asymptotical development a deformation of the natural symplectic form on the space $\J$ of almost complex structure of an integral symplectic manifold and a deformation of the scalar curvature moment map picture through the trace of the Berezin-Toeplitz operators. The star product behind Berezin-Toeplitz operators \cite{Schlich} is not equivalent, in general, to $*_{\nabla}$. Also, note that our construction does not required integrality of the symplectic form.

Charles \cite{Ch2} used the same framework as in \cite{FU} but considering metaplectic correction, which only holds under some topological restrictions. As in \cite{FU}, a bundle of quantum states is considered on $\J$ with canonical connection. Its curvature is computed and is shown to be a Toeplitz operator. The star product behind this construction is, due to the metaplectic correction, equivalent to $*_{\nabla}$. Pursuing this construction in the direction of Foth--Uribe's work one should get a similar moment map picture as we have obtained here.

In both paper \cite{FU} and \cite{Ch2}, Hamiltonian diffeomorphisms of the symplectic manifolds are lifted as endomorphisms of the quantum space bundle using the parallel transport. In both cases these lifts induces Hamiltonian automorphisms of the underlying star product, as in Proposition \ref{prop:Btparallel}

\item Shelukhin \cite{Shel} uses the scalar curvature moment map picture to define a quasi-morphism on the universal cover $\widetilde{\ham}(M,\omega)$ extending the corresponding Weinstein action homomorphism $\A^{\textrm{scal}}$. Its construction exploits an hyperbolicity property of the natural K\"ahler structure on $\J$. It is one of the ingredient of Charles' work \cite{Ch} studying asymptotic representations of Hamiltonian diffeomorphisms in which he obtained more general computations as Lemma \ref{lemme:AStheorem} and Proposition \ref{prop:Atildeholonomie} for path of Hamiltonian diffeomorphisms.

In our situation, the geometry of $(\E(M,\omega),\Omega^{\E})$ is, at first sight, flat. So that the construction of Shelukhin does not apply directly. But, it must be interesting to study if it is possible to extend the (formal) action homomorphism described in Subsection \ref{sect:formalaction} as (formal) quasi-morphism (with appropriate definition of formal quasi-morphism of course).

\item In \cite{Shel}, the Weinstein action homomorphism  $\A^{\textrm{scal}}$ as above, is shown to be equal to another invariant of symplectic geometry $I_{c_1}$ coming from the study of Hamiltonian fibrations attached to Hamiltonian loops \cite{LMP}. This identification implies the equality of $\A^{\textrm{scal}}$ with the original Futaki invariant \cite{Futor} which inspired the author for Corollary \ref{cor:futaki}. In \cite{Krav}, Kravchenko studied the Fedosov deformation quantization of Hamiltonian fibration. Can we derive from Kravchenko's work an invariant similar to $I_{c_1}$ that is equal to $\widetilde{\A}$?

\item In \cite{LLF}, we observe that for Fedosov star products $*_{\nabla, \nu f^*\chi}$, for $\chi$ being a fixed symplectic form and $f$ being a diffeomorphism of $M$, the trace is related at first order in $\nu$ to a moment map picture on the connected component of the space of diffeomorphisms $\textrm{Diff}_0(M)$. It should be interesting to consider a star product algebra bundle over $\textrm{Diff}_0(M)$ and to deform the moment map geometry of $\textrm{Diff}_0(M)$ in a similar vein as here.

We also observe that the scalar curvature moment map picture appears in the trace of the Bordemann--Waldmann star product \cite{BW}, which is very close to the star product underlying the work of Foth--Uribe discussed in 1. Surely, the work of this paper could be adapted to the Bordemann--Waldmann star product to deform the scalar curvature moment map picture and obtain directly the asymptotic results from \cite{FU}.

\end{enumerate}

\subsection*{Acknowledgements}

First, I would like to thank Simone Gutt for her encouragement in my research and for supporting my scientific career. I thank Akito Futaki for his interest in my work and many fruitful discussions. I thank Martin Schlichenmaier who give me the chance to do this work in the excellent research environment of the University of Luxembourg. I also thank Michel Cahen for his longstanding support.


\end{document}